\documentclass{article}
\usepackage{amsmath,amssymb,amsthm,graphicx,epsfig,float,url}
\usepackage{pdfsync}
\usepackage[usenames, dvipsnames]{xcolor}
\usepackage{tikz}
\usepackage{subfig}
\usepackage{bbm}
\usepackage{stmaryrd}
\usepackage{mathrsfs}  
\usepackage{caption}
\usepackage{float}
\usepackage[makeroom]{cancel}
\usetikzlibrary{patterns}
\newcommand{\R}{\textnormal{I\kern-0.21emR}}
\newcommand{\N}{\textnormal{I\kern-0.21emN}}

\newcommand{\A}{\mathcal{A}}

\renewcommand{\geq}{\geqslant}
\renewcommand{\leq}{\leqslant}

\def\B{{\mathbb B}}
\def\e{{\varepsilon}}
\allowdisplaybreaks

\usepackage[dvipsnames]{xcolor}

\newtheorem*{theorem*}{Theorem}
\newtheorem{conjecture}{Conjecture}  

\newtheorem{theorem}{Theorem}  
\newtheorem{proposition}{Proposition}

\newtheorem{definition}{Definition}
\newtheorem{lemma}{Lemma}
\newtheorem{claim}{Claim}

\theoremstyle{definition}\newtheorem{remark}{Remark}

\def\O{{\Omega}}
\def\n{{\nabla}}
\def\p{{\varphi}}

\usepackage{xargs}
 \usepackage[colorinlistoftodos,textsize=small]{todonotes}
 \newcommandx{\unsure}[2][1=]{\todo[linecolor=red,backgroundcolor=red!25,bordercolor=red,#1]{#2}}
 \newcommandx{\change}[2][1=]{\todo[linecolor=blue,backgroundcolor=blue!25,bordercolor=blue,#1]{#2}}
 \newcommandx{\info}[2][1=]{\todo[linecolor=green,backgroundcolor=green!25,bordercolor=green,#1]{#2}}
 \newcommandx{\improvement}[2][1=]{\todo[linecolor=yellow,backgroundcolor=yellow!25,bordercolor=yellow,#1]{#2}}
 
  \newcommandx{\biblio}[2][1=]{\todo[linecolor=blue,backgroundcolor=magenta!25,bordercolor=blue,#1]{#2}}

\begin{document}
\nocite{*}
\title{Quantitative inequality for the eigenvalue of a Schr\"{o}dinger operator in the ball }


\author{Idriss Mazari\footnote{Sorbonne Universit\'es, UPMC Univ Paris 06, UMR 7598, Laboratoire Jacques-Louis Lions, F-75005, Paris, France (\texttt{idriss.mazari@sorbonne-universite.fr})}}
\date{\today}

\maketitle

\begin{abstract}
The aim of this article is to  prove a  quantitative inequality for the first eigenvalue of a Schr\"{o}dinger operator in the ball. More precisely, we optimize the first eigenvalue $\lambda(V)$ of the operator $\mathcal L_v:=-\Delta-V$ with Dirichlet boundary conditions with respect to the potential $V$, under $L^1$ and $L^\infty$ constraints on $V$. The solution has been known to be the characteristic function of a centered ball, but this article aims at proving a sharp growth rate of the following form: if $V^*$ is a minimizer, then $\lambda(V)-\lambda(V^*)\geq C ||V-V^*||_{L^1(\O)}^2$ for some $C>0$.
\\The proof relies on two notions of derivatives for shape optimization:  parametric derivatives and shape derivatives. We use parametric derivatives to handle radial competitors, and shape derivatives to deal with normal deformation of the ball. A dichotomy is then established to extend the result to all other potentials. We develop a new method to handle radial distributions and a comparison principle to handle second order shape derivatives at the ball. Finally, we add some remarks regarding the coercivity norm of the second order shape derivative in this context.
\end{abstract}


\medskip

\\
\medskip

\tableofcontents
\section{Introduction}
\subsection{Structure of the paper}
In the first part (Section \ref{Ss:Setting}) of the introduction, we  lay out the mathematical setting of this article, the main results and give the relevant definitions. In the second part of the Introduction, we give bibliographical references concerning quantitative inequalities and the biological motivations this study stems from. We then explain, in Subsection \ref{Ss:StructureProof}, how the proof differs from that of other quantitative inequalities. 
\\The core of this paper is devoted to the proof of Theorem \ref{Th:Quanti}. In the conclusion, we give a conjecture and a final comment on a possible  way to obtain an optimal exponent  using parametric derivatives.
\subsection{Mathematical setting}\label{Ss:Setting}
The optimization of eigenvalues of elliptic operators defined on domains with a zero-order term (i.e with a potentiel) with respect to either the domain (with a fixed potential defined on a bigger domain) or the potential (with a fixed domain) is a classical question in optimization under partial differential equations constraints. The example under scrutiny here is the operator 
\begin{equation}\label{Eq:Operator}\mathcal L_V:u\in W^{2,2}_0(\O)\mapsto-\Delta u-Vu,\end{equation} where $\O$ is a smooth domain. Under the assumption that $V\in L^\infty(\O)$, this operator is known to have a first, simple eigenvalue, denoted by $\lambda(V)$, and associated with an eigenfunction $u_V$, which solves 
\begin{equation}\label{Eq:EigenFunction}
\left\{\begin{array}{ll}
-\Delta u_V-Vu_V=\lambda(V)u_V\text{ in }\O,&
\\u_V=0\text{ on }\partial \O.&
\\\int_\O u_V^2=1.
\end{array}\right.
\end{equation}
Alternatively, this eigenvalue admits the following variational formulation in terms of Rayleigh quotients:
\begin{equation}\label{Eq:Rayleigh}
\lambda(V)=\inf_{u\in W^{1,2}_0(\O)\, , \int_\O u^2=1}\left\{\int_\O |\n u|^2-\int_\O V u^2\right\}.\end{equation}
We make stronger assumptions on the potential $V$ and require that it lies in 
\begin{equation}\label{De:Admissible}
\mathcal M(\O):=\left\{V\in L^\infty(\O)\, , 0\leq V\leq 1\, , \frac1{|\O|}\int_\O V=V_0\right\},\end{equation}
where $V_0>0$ is a real parameter such that
$$V_0<1$$ so that $\mathcal M(\O)$ is non-empty.
The optimization problem we focus on in this paper is
\begin{equation}\label{Eq:PvG}
\inf_{V\in \mathcal M(\O)}\lambda(V).\end{equation}
This problem has drawn a lot of attention from the mathematical community over the last decades, and is quite a general one. It is particularly relevant in the context of mathematical biology, see Section \ref{Su:Biology}.

\subsection{Main results}

\subsubsection{Notations}
Here and throughout, the underlying domain is $\O=\mathbb B(0;R)$. 
\\The parameter $r^*$ is chosen so that 
$$V^*:=\chi_{\mathbb B(0;r^*)}=\chi_{E^*}\in \mathcal M(\O).$$
The constant $c_n>0$ is the $(n-1)$-dimensional volume of the unit sphere in dimension $n$.
\\For any $E\subset \B$, we define $\lambda(E):=\lambda(\chi_E)$.

\subsubsection{Quantitative inequality}
A classical application of Schwarz' rearrangement  shows  that $V^*$ is the unique minimizer of $\lambda$ in $\mathcal M(\O)$:
$$\forall V\in \mathcal M(\O)\, , \lambda(V)\geq \lambda(V^*).$$
We refer to \cite{Kawohl} for an introduction to the Schwarz rearrangement and to \cite{LamboleyLaurainNadinPrivat} for its use in this context. For the sake of completeness, we also prove this result in Annex \ref{An:Schwarz}.
\\The goalf of this paper is to establish the following quantitative spectral inequality:
\begin{theorem}\label{Th:Quanti} Let $n=2$.
There exists a constant $C>0$ such that, for any $V\in \mathcal M(\O)$, there holds
\begin{equation}\label{Eq:Quantitative}
\lambda(V)-\lambda(V^*)\geq C ||V-V^*||_{L^1}^2.\end{equation}\end{theorem}
\color{black}
\begin{remark}
The hypothesis $n=2$ is essentially needed in Step 4 of the proof, where we gather estimates on second order shape derivatives and thus need to diagonalise the shape Hessian using Fourier series.  This part of the proof can be adapted to dimension 3 using spherical harmonics, but we choose to present it in this low dimension context for the sake of readability. Furthermore, in Step 5, we need convergence of the eigenfunctions in H\"{o}lder spaces and can obtain it in dimensions 2 and 3 in a straighforward manner.
\end{remark}
\color{black}
We can rephrase this result in terms of Fraenkel asymmetry: as we will explain in Subsection \ref{Su:Quanti}, this is the natural property  to expect in the context of quantitative inequalities. Indeed, if $V=\chi_E\in \mathcal M(\O)$, then 
$$||V-V^*||_{L^1}=\left|E^*\Delta E\right|$$ where $\Delta$ stands for the symmetric difference, and, if we define the fraenkel asymmetry of $E$ as 
$$\mathcal A(E):=\inf_{\mathbb B(x;r)\, , \chi_{\B(x;r)}\in \mathcal M(\O)}|E\Delta \B(x;r)|$$ then it follows from Theorem \ref{Th:Quanti} that 
$$\lambda(\chi_E)-\lambda(\chi_{E^*}) \geq C \mathcal A(E)^2.$$
We thus have a parametric version of a quantitative inequality with what is in fact a sharp exponent.

\subsubsection{A comment on parametric and shape derivatives}\label{Su:Difference}
The proof of Theorem \ref{Th:Quanti} relies on parametric and shape derivatives, and the aim of this Section is to give possible links between the two notions and, most notably, to give a situation where this link is no longer possible. This is obviously in sharp contrast with classical shape optimization, since here we are only optimizing with respect to the potential, while it is customary to derive Faber-Krahn type inequalities, i.e to optimize with respect to the domain $\O$ itself, see Subsection \ref{Su:Quanti}.
\\Roughly speaking, there are two ways to tackle spectral optimizatoin problems such as \eqref{Eq:PvG}: the parametric approach and the shape derivative approach. By parametric approach we mean the following: 
\begin{definition}
We define, for any $V\in \mathcal M(\O)$, the tangent cone to $\mathcal M(\O)$ at $V$ as
$$\mathcal T_V:=\left\{h:\O\rightarrow [-1;1]\, , \forall \e\leq 1\, ,  V+\e h\in \mathcal M(\O)\right\}$$ and define, provided it exists, the parametric  derivative of $\lambda$ at $V$ in the direction $h \in \mathcal T_V$ as 
$$\dot \lambda (V)[h]:=\lim_{t\to 0}\frac{\lambda(V+th)-\lambda(V)}t.$$ \end{definition}
In this case, the optimality condition reads
$$\forall h \in \mathcal T_V\, , \dot \lambda(V)[h]\geq 0.$$ In the conclusion of this article, we will explain why there holds
$$\forall h \in \mathcal T_{V^*}\, , \dot \lambda (V^*)[h]\geq C||h||_{L^1}^2.$$
Since, for any $V \in \mathcal M(\O)$, $h:=V-V^*\in \mathcal T_{V^*}$, what we get is 
$$\dot \lambda(V^*)[V-V^*]\geq C ||V-V^*||_{L^1}^2$$ that is, an infinitesimal version of Estimate \eqref{Eq:Quantitative}, and one might wonder whether such infinitesimal estimates might lead to global qualitative inequality of the type \eqref{Eq:Quantitative}. This is customary, in the context of shape derivatives. We need to define the notion of shape derivative before going further: 
\begin{definition}\label{De:ShapeStab}
Let $\mathcal F:E\mapsto \mathcal F(E)\in \R$ be a shape functional. We define
$$\mathcal X_1(V^*):=\left\{\Phi:\B(0;R)\rightarrow \R^2\, ,||\Phi||_{W^{1,\infty}}\leq 1\, , \forall t \in(-1;1)\, , \chi_{(Id+t\Phi)(\B^*)}\in \mathcal M(\O).\right\}$$as the set of admissible perturbations at $E^*$.
The shape derivative of first (resp. second) order of a shape function $\mathcal F$ at $V^*$ in the direction $\Phi$ is 
\begin{multline}
\mathcal F'(E^*)[\Phi]=\lim_{t\to 0}\frac{\mathcal F\Big((Id+t\Phi)E^*\Big)-\mathcal F(E^*)}t
\\\text{ (resp.}\mathcal F''(E^*)[\Phi,\Phi]:=\lim_{t^2\to 0}\frac{\mathcal F\Big((Id+t\Phi)E^*\Big)-\mathcal F(E^*)-\mathcal F'(E^*)(\Phi)}{t^2}.\text{)}\end{multline}
\end{definition}
A customary way to derive quantitative inequality is to show that, at a given shape $E$, there holds
$$\mathcal F'(E)[\Phi]=0\, , \mathcal F''(E)[\Phi,\Phi]>0$$ and to lift the last inequality to a quantitative inequality of the form
$$\mathcal F'(E)[\Phi,\Phi]\geq C||\Phi||_s^2$$ where $||\cdot||_s$ is a suitable norm; we refer to \cite{DambrineLamboley} for more details but for instance one might have $||\left.\Phi\right|_{\partial E}||_{L^1}^2$ which often turns out to be the suitable exponent for a quantitative inequality. This quantitative inequality for shape deformations is usually not enough, and we refer to Section \ref{Su:Quanti} for more details and bibliographical references.

\subsubsection{A remark on the proof}
 The main innovation of this paper is the proof of Theorem \ref{Th:Quanti} which, although it uses shape derivatives as is customary while proving quantitative inequalities, see Subsection \ref{Su:Quanti}, relies  heavily on parametric derivatives. This is allowed by the fact that we are working with a potential defined on the interior of the domain.
 \\Furthermore, we also prove that, unlike classical shape optimization, our coercivity norm for the second order shape derivative is the $L^2$ norm.

\subsection{Bibliographical references}
\subsubsection{Quantitative spectral inequalities}\label{Su:Quanti}
\vspace{0.2cm}
\paragraph{Spectral deficit for Faber-Krahn type inequalities:}
Quantitative spectral inequalities have received a lot of attention for a few decades, and are usually set in a context which is more general than the one introduced here. The main goal of such inequalities were to derive quantitative versions of the Faber-Krahn inequality: for a given parameter $\beta\in (0;+\infty]$ and a bounded domain $\O\subset \R^n$, consider the first eigenvalue $\eta_\beta(\O)$ of the Laplacian with Robin boundary conditions:
\begin{equation}\label{Eq:LaplaceRobin}
\left\{\begin{array}{ll}
-\Delta u_{\beta,\O}=\eta_\beta(\O)u_{\beta,\O}\text{ in }\O,&
\\\frac{\partial u_{\beta,\O}}{\partial \nu}+\beta u_{\beta,\O}=0\text{ on }\partial \O,&
\end{array}\right.
\end{equation} 
with the convention that $\eta_\infty(\O)$ is the first Dirichlet eigenvalue of the Laplace operator. Although it has been known since the independent works of Faber \cite{Faber} and Krahn \cite{Krahn} that, whenever $\O^*$ is a ball with the same volume as $\O$, there holds
$$\eta_\infty(\O)\geq \eta_\infty(\O^*),$$
the question of providing a sharp lower bound for the so-called \textit{spectral deficit} $$R_\infty(\O)=\eta_\infty(\O)-\eta_\infty(\O^*)$$ remained largely open until  Nadriashvili and Hansen \cite{NadirashviliHansen} and Melas, \cite{Melas}, using Bonnesen type inequalities, obtained a lower bound  on the spectral deficit involving  quantities related to the geometry of domain $\O$ through the inradius. In a later work, Brasco, De Philippis and Velichkov, \cite{BDPV}, the sharp version of the quantitative inequality, namely:
\begin{equation}R_\infty(\O)\geq C \mathcal A(\O)^2.\end{equation} Their proof relies uses as a first step a second order shape derivative argument, a series of reduction to small asymmetry regime and, finally, a quite delicate selection principle. We comment in the next paragraph on 
the role of second order shape derivative for generic quantitative inequalities and the part it plays in our proof.

For a survey of the history and proofs of quantitative Faber-Krahn inequalities, we refer to the survey \cite{BrascoDePhilippis} and the references therein. 

 For quantitative versions of spectral inequalities with general Robin boundary conditions, the Bossel-Daners inequality, first derived by Bossel in dimension 2 in \cite{Bossel} and later extended by Daners in all dimensions in \cite{Daners} reads:
$$R_\beta(\O):=\eta_\beta(\O)-\eta_\beta(\B)\geq 0\, , \beta>0,$$
and a quantitative version of the Inequality was proved by Bucur, Ferone, Nitsch and Trombetti in \cite{BFNT}:
$$R_\beta(\O)\geq C \mathcal A(\O)^2.$$ Their method for $\beta<\infty$ is different from the case of Dirichlet eigenvalue and relies on a free boundary approach.

\color{black}
\paragraph{Quantitative estimates for optimal potentials}
In the parametric context, that is, when optimising a criterion with respect to a potential, two references whose results are related to the one of the present paper are \cite{BrascoButtazzo} and \cite{CarlenLieb}; in both these papers, the $L^\infty$ constraint $0\leq V\leq 1$ we consider in the present paper is not considered and they mainly deal with $L^p$ constraints. Namely, in \cite{CarlenLieb}, the main result, in the two dimensional case, is the following: consider, for a parameter $\gamma>0$, a  non-postivie potential $V\in L^{1+\gamma}(\R^2)\,, V\leq 0$, the operator 
$$\mathfrak L_V:=-\Delta+V$$ and its first eigenvalue 
$$\mathfrak E(V):=\inf\left\{\int_{\R^2} |\n \psi|^2+\int_{\R^2} V \psi^2\,, \psi \in W^{1,2}(\R^2)\,, \int_{\R^2}\psi^2=1\right\}.$$
The relevant optimisation problem is then 
$$C_{\gamma}:=\sup\left\{\frac{|\mathfrak E(V)|}{\left(\int_{\R^2}V^{\gamma+\frac12}\right)^{\frac1\gamma}}\right\},$$
and the optimal class is 
$$\mathfrak M_\gamma:=\left\{V\,, \frac{|\mathfrak E(V)|}{\left(\int_{\R^2}V^{\gamma+\frac12}\right)^{\frac1\gamma}}=C_\gamma\,, V \leq 0\right\}.$$ The first result of \cite{CarlenLieb} is that $\mathfrak M_\gamma$ is non-empty. Furthermore they obtain many  stability estimates \cite[Theorem 2.2]{CarlenLieb}, one of which reads: there exists $c_\gamma>0$ such that, for any $\gamma>0$ such that $\gamma\leq \frac32$, for any $V\in L^{\gamma+\frac12}$, $V\leq 0$, 
\begin{equation}
\frac{|\mathfrak E(V)|}{\left(\int_{\R^2}V^{\gamma+1}\right)^{\frac1\gamma}}\leq C_\gamma-c_\gamma \inf_{W \in \mathfrak M_\gamma}\frac{\Vert V-W\Vert_{\gamma+1}}{\Vert V\Vert_{\gamma+1}^2}.\end{equation}
Their proof is strongly related to the stability of functional inequalities, such as the Gagliardo-Nirenberg-Sobolev or H\"{o}lder inequalities, and it is not clear to us that their methods can be used in the context we are presently considering.

In \cite{BrascoButtazzo}, a stability estimate for the Dirichlet energy with respect to the potential is obtained. One of their main theorems reads as follows \cite[Theorem B]{BrascoButtazzo}: let $\O$ be a smooth  domain in $\R^n$ and $f\in W^{-1,2}(\O)$. Let, for any potential $V$ enjoying some suitable integrability properties, $u_V\in W^{1,2}_0(\O)$ be the solution of
$$-\Delta u_V+V u_V=f$$ and define the associated Dirichlet energy
$$\mathcal E_f(V):=-\frac12 \int_\O |\n u_V|^2-\frac12\int_\O V u_V^2.$$ Define, for any $p\in (1;\infty)$, 
$$\mathfrak V_p:=\left\{V\in L^p(\O),\Vert V\Vert_{L^p}\leq 1\right\}.$$ In \cite{Buttazzo2014}, it was proved that a solution $V_0$ to the maximisation problem
$$\sup_{V\in \mathfrak V_p}\mathcal E_f(V)$$ exists, is unique and satisfies $\Vert V_0\Vert_{L^p}=1$. Then \cite[Theorem B]{BrascoButtazzo} reads: there exists $\sigma_p>0$ such that 
$$\forall V \in \mathfrak V_p\,, \mathcal E_f(V)\leq \mathcal E_f(V_0)-\sigma_p ||V-V_0||_p^2.$$
Here, the proof of this stability results relies on stability estimates for functional inequalities, and the optimality conditions are quite different from ours, as the optimality system in their case is a non-autonomous, semilinear equation. It is not clear to us that the constraints $0\leq V\leq 1$ we consider here, as well as the spectral quantity we are optimising, can be handled through the methods of \cite{BrascoButtazzo}.

\color{black}

\paragraph{The role of second order shape derivatives:} We only want to mention here the results we draw our inspiration from in Step 4 of the proof, and do not aim at giving out the rigorous mathematical setting of the results mentioned below. We refer to \cite{DambrineLamboley} for a thorough presentation of the link between second order shape derivatives, local shape stability and local quantitative inequalities.

As we said in the previous paragraph, most proofs of quantitative inequalities start with a local quantitative inequality for shape perturbation of the optimum $E_1$: namely, if $\mathcal F$ is a regular enough shape functional and $E_1$ is an admissible set such that 
\begin{equation}\label{Eq:Condition}\mathcal F (E_1)[\Phi]=0\, , \mathcal F''(E_1)[\Phi,\Phi]>0\end{equation} for any $\Phi \in \mathcal X_1(E_1)$ then it is proved in \cite{Dambrine}, \cite{DambrinePierre} that $E_1$ is a strict local minimizer in a $\mathscr C^{2,\alpha}$ neighbourhood of $E$ (actually, in these two articles, the authors assume a coercivity of the second order derivative in in $H^{\frac12}$ norm on $\Phi$). In \cite{DambrineLamboley}, Dambrine and Lamboley proved that the same conditions imply a local quantitative inequality under certain technical assumptions. Roughly speaking,  their result implies the following result: Condition \eqref{Eq:Condition} implies that if, for any function $h\in \mathscr C^0(\partial \O)\cap W^{2,\infty}(\partial \O)$ we define $E_1^h$ as the domain bounded by 
$$\partial E_1^h:=\left\{x+h(x)\nu_{E_1}(x)\, , x\in \partial E_1\right\}$$ then there exists $\ell>0$ such that, for any $||h||_{W^{2,\infty}(\partial \O)}\leq \ell$ there holds
$$\mathcal F(E_1)+C||h||_{L^1}^2\leq\mathcal F(E_1^h)$$ for some $C>0$.
\\Their result actually holds in stronger norm, but this is the version we wanted to mention here since it is the one we will adapt in our parametric setting. We note however, that the authors, in their problems, prove their inequalities for shape using a $H^{\frac12}$ coercivity norm for the second derivative: they usually have 
$$\mathcal F(E_1)''[h,h]\geq C||h||_{H^{\frac12}}^2.$$In this expression, we have defined $\mathcal F''(E_1)[h,h]:=\mathcal F(E_1)''[x+h\nu_{E_1}(x),x+h\nu_{E_1}(x)]$. Here, in Step 4, we will show, using comparison principles, that the optimal coercivity norm is the $L^2$ norm.
\paragraph{Difference with our proofs and contribution:}
\color{black}
As mentioned in the previous paragraphs, most of the existing literature deals with stability estimates in the context of shape optimisation or, when dealing with optimisation with respect to the potential, with $L^p$, $1<p<\infty$ constraints. We believe this article to be a first step in the context of $L^\infty$ constraints. In this setting, other types of phenomenons  appear.  As we will  see while proving Theorem \ref{Th:Quanti}, a shape derivative approach can not be sufficient  in of its own for our purposes, as is usually the case, but is needed. Indeed, changes in the topology of competitors may occur, which calls for a new specific method.  To tackle the second order shape derivative, we use a comparison principle. 
\color{black}

\subsubsection{Mathematical biology}\label{Su:Biology}
We briefly sketch some of the biological motivations for the problem under scrutiny here. Following the works of Fisher, \cite{Fisher}, Kolmogoroff, Petrovsky and Piscounoff \cite{KPP}, a popular model for population dynamics in a bounded domain is the following so-called logistic-diffusive equation: 
\begin{equation}\label{dE}
\left\{\begin{array}{ll}
\frac{\partial u}{\partial t}=\Delta u+u(m-u)\text{ in }\O\, ,&
\\u=0\text{ on } \partial \O\,,&
\\u(t=0)=u_0\geq 0\, , u_0\neq 0.&
\end{array}
\right.
\end{equation}
In this equation, $m\in L^\infty(\O)$ accounts for the spatial heterogeneity and can be interpreted in terms of resources distribution: the zones $\{m\geq 0\}$ are favorable to the growth of the population, while the zones $\{m\leq 0\}$ are detrimental to this population. The particular structure of the non-linearity $-u^2$ (which accounts for the Malthusian growth of the population) makes it so that two linear steady states equations are relevant to our study: the steady-logistic diffusive equation
\begin{equation}
\left\{\begin{array}{ll}
\Delta \theta+\theta(m-\theta)=0\text{ in }\O\, ,&
\\\theta=0\text{ on } \partial \O\,,&
\\\theta\geq 0,&
\end{array}
\right.
\end{equation}
and the first eigenvalue equation of the linearization of \eqref{dE} around the solution $z\equiv 0$:
\begin{equation}\label{Eq:SLDE}
\left\{\begin{array}{ll}
-\Delta \p_m-m\p_m=\lambda(m)\p\text{ in }\O\, ,&
\\\p=0\text{ on } \partial \O,&
\end{array}
\right.
\end{equation}
where $\lambda(m)$ is the first eigenvalue of the operator $\mathcal L_m$ defined in \eqref{Eq:Operator}. More precisely, it is known (see \cite{BHR,CantrellCosner1,Skellam}) that 
\begin{enumerate}
\item Whenever $\lambda(m)<0$,  \label{Eq:SLDE} has a unique solution $\theta_m$, and any solution $u=u(t,x)$ of \eqref{dE} with initial datum $u_0\geq 0\,, u_0\neq 0$ converges in any $L^p$ to $\theta_m$ as $t\to \infty$.
\item Whenever $\lambda(m)\geq0$, any solution $u=u(t,x)$ of \eqref{dE} with initial datum $u_0\geq 0\,, u_0\neq 0$ converges in any $L^p$ to $0$ as $t\to \infty$.
\end{enumerate}
The eigenvalue which we seek to minimize can thus be interpreted as a measure of the survival ability given by a resources distribution, and later works investigated the problem of minimizing $\lambda(m)$ with respect to $m$ under the constraint $m\in \mathcal M(\O)$, where $\mathcal M(\O)$ is defined in \eqref{De:Admissible}. In other words, this is the problem \eqref{Eq:PvG}.
\\In the case of Neumann boundary conditions, Berestycki, Hamel and Roques introduced the use of a rearrangement (due to Berestycki and Lachand-Robert,\cite{BLR}) in that context, see \cite{BHR} and \cite{Kawohl} for an introduction to rearrangement, and further geometrical properties of optimizers were derived by Lou and Yanagida, \cite{LouYanagida}, by Kao, Lou and Yanagida \cite{KaoLouYanagida}. We do not wish to be exhaustive regarding the literature of this domain and refer to \cite{LamboleyLaurainNadinPrivat} where Lamboley, Laurain, Nadin and Privat investigate several properties of solutions of \eqref{Eq:PvG} under a variety of boundary conditions, and the references therein.

\section{Proof of Theorem \ref{Th:Quanti}}\label{Pr:Quanti}

\subsection{Background on \eqref{Eq:PvG} and structure of the proof}\label{Ss:StructureProof}
We recall that we work in $\O=\mathbb B(0;R)\subset \R^2$, that 
$$\mathcal M:=\mathcal M\big(\B(0;R)\big)=\left\{0\leq V \leq 1\, , \int_\O V=V_0\right\}$$
and that $r^*$ is chosen so that 
$$|\B (0;r^*)|=V_0$$ i.e such that 
$$V^*=\chi_{\mathbb B(0;r^*)}\in \mathcal M.$$ We define 
$$\mathbb S^*:=\partial \B^*.$$ We first recall the following simple consequence of Schwarz' rearrangement:
\begin{lemma}\label{Le:Schwarz}
$V^*$ is the unique minimizer of $\lambda$ in $\mathcal M$: for any $V \in \mathcal M$, $$\lambda(V^*)<\lambda(V).$$ The associated eigenfunction $u_*$ is decreasing and radially symmetric.\end{lemma}
This result is well-known, but for the sake of completeness we prove it in Annex \ref{An:Schwarz}.
\\The proof of Theorem \ref{Th:Quanti} relies on the study of two auxilliary problem: we introduce, for a given $\delta>0$, the new admissible sets 
\begin{equation}\label{De:AdmissibleDelta}\mathcal M_\delta:=\left\{V \in \mathcal M\, ,||V-V^*||_{L^1}=\delta\right\},\end{equation}
\begin{equation}\label{De:AdmissibleDeltaRadial}\tilde{\mathcal M}_\delta:=\left\{V\text{ radially symmetric, }V \in \mathcal M\, ,||V-V^*||_{L^1}=\delta\right\}\end{equation}
 and study the two variational problems
\begin{equation}\label{Eq:PvDeltaRadial}
\inf_{V \in \tilde{\mathcal M_\delta}}\lambda(V)\end{equation}
and 
\begin{equation}
\label{Eq:PvDelta}\inf_{V \in \mathcal M_\delta}\lambda(V).\end{equation} Obviously, Theorem \ref{Th:Quanti} is equivalent to the existence of $C>0$ such that
\begin{equation}\label{Eq:Rate}\forall \delta>0\, ,\forall V \in \mathcal M_\delta\, , \lambda(V)-\lambda(V^*)\geq C \delta^2.\end{equation} 
\begin{remark}
This is a parametric version of the selection principle of \cite{AcerbiFuscoMoreni}, that was developed in \cite{BDPV}. We refer to \cite{BrascoDePhilippis} for a synthetic presentation of this selection principle. We note however that the fact that they use a perimeter constraint enables them to prove that a solution to their auxiliary problem is a normal deformation of the optimal shape. The main difficulty in the analysis of \cite{BDPV} is establishing $\mathscr C^2$ bounds for this normal deformation. Here, working with subsets as shape variables gives, from elliptic regularity, enough regularity to carry out this step when the solution of the auxiliary problem is a normal deformation of $\B^*$. However, we conjecture that the solutions of \eqref{Eq:PvDeltaRadial} and \eqref{Eq:PvDelta} are equal and are disconnected (see Step 3 and the Conclusion for a precise conjecture), so that the core difficulty is proving  that handling the inequality for normal deformations and for radial distributions is enough to get the inequality for all other sets.
\end{remark}To prove \eqref{Eq:Rate}
we follow the steps below:
\begin{enumerate}
\item We first show that \eqref{Eq:PvDelta} and \eqref{Eq:PvDeltaRadial} have solutions. The solutions of \eqref{Eq:PvDelta} will be denoted by $\mathcal V_\delta$, the solution of \eqref{Eq:PvDeltaRadial} will be denoted by $\mathcal H_\delta$.
\item We prove that it suffices to establish \eqref{Eq:Rate} for $\delta$ small enough.
\item For \eqref{Eq:PvDeltaRadial} we fully characterize the solutions for $\delta>0$ small enough and prove that 
\begin{equation}\label{Eq:RadialRate}\forall \delta>0\, ,\forall V \in \Tilde{\mathcal M}_\delta\, , \lambda(V)-\lambda(V^*)\geq C \delta^2.\end{equation} In other words, we prove that Theorem \ref{Th:Quanti} holds for radially symmetric functions. 
\item We compute the first and second order shape derivatives of the associated Lagrangian at the ball $\B^*$ and prove a $L^2$-coercivity estimate for the second order derivative. We comment upon the fact that (unlike many shape optimization problems) this is the optimal coercivity  norm at the beginning of this Step.
 We use  this information to  prove that Theorem \ref{Th:Quanti} holds for domains that are small normal deformations of $\B^*$ with bounded mean curvature.
\item We establish a dichotomy for the behaviour of $\mathcal V_\delta$ and prove that \eqref{Eq:Rate} holds for any $V$ by using \eqref{Eq:RadialRate} and Step 4. 
\end{enumerate}

\subsection{Step 1: Existence of solutions to \eqref{Eq:PvDeltaRadial}-\eqref{Eq:PvDelta}}\label{Su:Step1}
We prove the following Lemma:
\begin{lemma}\label{Le:Existence}
The optimization problems \eqref{Eq:PvDeltaRadial} and \eqref{Eq:PvDelta} have solutions.\end{lemma}
\begin{proof}[Proof of Lemma \ref{Le:Existence}]
The proof follows from the following claim:
\begin{claim}\label{Cl:Compact}
$\mathcal M_\delta$ and $\tilde{ \mathcal M}_\delta$ are compact for the weak $L^\infty$ topology.\end{claim}
 We postpone the proof to the end of this Proof.
\\Lemma \ref{Le:Existence} follows from this claim, and we only write the details for \eqref{Eq:PvDelta}. Let $\{V_k\}_{k\in \N}$ be a minimizing sequence for $\lambda$ in $\mathcal M_\delta$.   From Claim \ref{Cl:Compact}, there exists $V_\infty\in \mathcal M_\delta$ such that  $$V_k\rightharpoonup V_\infty.$$ The notation $\rightharpoonup$ stands for the weak convergence in the  weak $L^\infty$-* sense.
Let, for any $k\in \N$, $u_k:=u_{V_k}$ and $\lambda_k:=\lambda(V_k)$.
\\We first note that the sequence $\{\lambda_k\}_{k\in \N}$ is bounded. Indeed, let $\p\in W^{1,2}_0(\O)$ be such that 
$\int_\O \p^2=1.$ From the formulation in terms of Rayleigh quotients  and $V_k\geq 0$  there holds
$$\lambda_k\leq \int_\O |\n \p|^2-\int_\O V_k\p^2\leq \int_\O |\n \p|^2.$$ This gives an upper bound. For a lower bound, let $\lambda_1(\O)$ be the first Dirichlet eigenvalue of $\O$ (equivalently, this is the eigenvalue associated with $V=0$). From $V\leq 1$, $\int_\O u_k^2=1$ and the variational formulation for $\lambda_1(\O)$ there holds
$$\lambda_1(\O)-1\leq \int_\O |\n u_k|^2-\int_\O Vu_k^2\leq \lambda_k$$ so that the sequence also admits a lower bound.
 It is straightforward to see that $\{u_k\}_{k\in \N}$ is bounded in $W^{1,2}_0(\O)$ so that from the Rellich-Kondrachov Theorem there exists $u_\infty\in W^{1,2}_0(\O)$ such that $\{u_k\}_{k\in \N}$ converges strongly in $L^2$ and weakly in $W^{1,2}_0(\O)$ to $u_\infty$. Passing to the limit in the weak  formulation 
$$\forall v \in W^{1,2}_0(\O)\, , \int_\O \langle \n u_k,\n v\rangle-\int_\O V_ku_kv=\lambda(V_k)\int_\O u_kv,$$ in the normalization condition 
$$\int_\O u_k^2=1$$ and in 
$$u_k\geq 0$$ readily shows that $u_\infty$ is a non-trivial eigenfunction of $\mathcal L_{V_\infty}$. Furthermore, it is non-negative. Since the first eigenfunction is the only eigenfunction with a constant sign, this proves that $\lambda_\infty=\lambda(V_\infty)$ and that $u_\infty$ is the eigenfunction associated with $V_\infty$. Thus:
$$\lambda(V_\infty)=\inf_{V\in \mathcal M_\delta}\lambda(V).$$
It remains to prove Claim \ref{Cl:Compact}:

\begin{proof}[Proof of Claim \ref{Cl:Compact}] We only prove it for $\mathcal M_\delta$.
\\Let $\{V_k\}_{k\in \N}\in \mathcal M_\delta^\N$. We define, for any $k\in \N$
$$h_k:=V_k-V^*.$$
Since $V^*=\chi_{\B^*}$ and $0\leq V_k\leq 1$, the following signe conditions hold on $h_k$:
\begin{equation}\label{Eq:Sign}h_k\geq 0\text{ in }(\B^*)^c\, , h_k\leq 0\text{ in }\B^*.\end{equation}
Since $\int_\O V_k=\int_\O V^*$ there holds
\begin{equation}\label{Eq:Bal}\int_{\B^*}h_k=-\int_{(\B^*)^c} h_k.\end{equation}
Finally from 
$$||V_k-V^*||_{L^1}=\delta$$  there comes
\begin{align}
\delta&=\int_\O |V_k-V^*|
\\&=\int_\O|h_k|
\\&=\int_{(\B^*)^c} h_k-\int_{\B^*}h_k\text{ from \eqref{Eq:Sign}}
\\&=-2\int_{\B^*}h_k\text{ from \eqref{Eq:Bal}}
\\\label{Eq:H+}&=2\int_{(\B^*)^c}h_k.
\end{align}
We see $h_k\chi_{(\B^*)^c}$ as an element of $L^\infty\left((\B^*)^c\right)$. Let $h_\infty^+$ be a weak-$L^\infty$ closure point of $\{h_k\}_{k\in \N}$ in $L^\infty\left((\B^*)^c\right)$.  From \eqref{Eq:H+} and \eqref{Eq:Sign} we have 
\begin{equation}\label{Conclus1}1\geq h_\infty^+\geq 0\, , \int_{(\B^*)^c} h_\infty^+=\frac{\delta}2.\end{equation}
For the same reason, there exists $h_\infty^-\in L^\infty(\B^*)$ such that
$$-1\leq h_\infty^-\leq 0\, , \int_{\B^*} h_\infty^-=-\frac{\delta}2$$ and
\begin{equation}\label{Conclus2}h_k\rightharpoonup h_\infty\text{ weak-$*$ in }L^\infty(\B^*).\end{equation}
We define 
$$h_\infty:=h_\infty^-\chi_{(\B^*)^c}+h_\infty^-\chi_{\B^*}$$ and it is clear that
$$h_k\rightharpoonup h_\infty \text{ weak-$*$ in }L^\infty(\O).$$
Setting $V_\infty:=V^*+h_\infty$ there holds
$$V_k\rightharpoonup V_\infty$$ and, by \eqref{Conclus1},\eqref{Conclus2}, $V_\infty \in \mathcal M_\delta(\O)$.
\end{proof}
\end{proof}

\subsection{Step 2: Reduction to small $L^1$ neighbourhoods of $V^*$}
We now prove the following Lemma:
\begin{lemma}\label{Le:L1Neigh}
To prove Theorem \ref{Th:Quanti}, it suffices to prove \eqref{Eq:Rate} for $\delta$ small enough, in other words it suffices to prove that there exists $C>0$ such that
$$\underset{\delta \to 0}{\lim\inf} \left(\inf_{V\in \mathcal M_\delta}\frac{\lambda(V)-\lambda(V^*)}{\delta^2}\right)\geq C.$$
\end{lemma}
\begin{proof}[Proof of Lemma \ref{Le:L1Neigh}]
We define, for any $V\neq V^*$,
$$G(V):=\frac{\lambda(V)-\lambda(V^*)}{||V-V^*||_{L^1}^2}$$ and consider a  minimizing  sequence $\{V_k\}_{k\in \N}$ for $G$. Then we either have, up to subsequence, 
$$\forall k \in \N\, ,||V_k-V^*||_{L^1}\geq \e>0$$ or 
$$||V_k-V^*||_{L^1}\underset{k\to \infty}\rightarrow 0.$$
In the first case, up to a converging subsequence, $V_k\rightharpoonup V_\infty$ in a weak $L^\infty$-$*$ sense and, by the same arguments as in the Proof of Claim \ref{Cl:Compact}, 
$$||V_\infty-V^*||_{L^1}\geq \e>0.$$ Furthermore, by the same arguments as in the proof of Lemma \ref{Le:Existence}, 
$$\lambda(V_k)\underset{k\to \infty}\longrightarrow \lambda(V_\infty)$$ so that 
$$G(V_k)\underset{k\to \infty}\longrightarrow \frac{\lambda(V_\infty)-\lambda(V^*)}{||V_\infty-V^*||_{L^1}}$$ and, by Lemma \ref{Le:Schwarz}, $G(V_\infty)>0$. Hence we only need to study the case $||V_k-V^*||_{L^1}\underset{k\to \infty}\rightarrow 0$, as claimed.
\end{proof}
The same arguments yield the following Lemma:

\begin{lemma}\label{Le:L1NeighRadial}
\eqref{Eq:RadialRate} is equivalent to proving that there exists $C>0$ such that
$$\underset{\delta \to 0}{\lim\inf} \left(\inf_{V\in\tilde{ \mathcal M}_\delta}\frac{\lambda(V)-\lambda(V^*)}{\delta^2}\right)\geq C.$$
\end{lemma}

\subsection{Step 3: Proof of \eqref{Eq:RadialRate}}
In this Subsection we prove \eqref{Eq:RadialRate} or, in other words, we prove that Theorem \ref{Th:Quanti} holds for radial distributions.
\begin{proof}[Proof of \eqref{Eq:RadialRate}]
We recall that, by Lemma \ref{Le:Existence}, there exists a solution to \eqref{Eq:PvDeltaRadial}. Let $\mathcal H_\delta$ be such a minimizer. 
\\We first characterize $\mathcal H_\delta$ for $\delta$ small enough. Let 
\begin{equation}\label{Eq:ADelta}\mathbb A_\delta:=\left\{|x|\leq r^*+r_\delta'\right\}\backslash\left\{x\, , r^*-r_\delta\leq |x|\leq r^*+r_\delta'\right\}\end{equation} be the annular structure such that \begin{equation}\label{Eq:ADeltaAd}\chi_{\mathbb A_\delta}\in \tilde {\mathcal M}_\delta.\end{equation} We represent it below
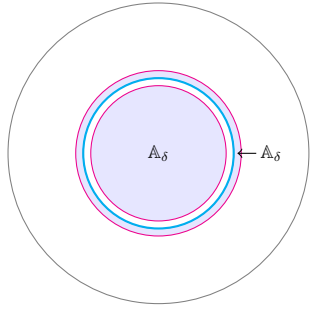
\begin{figure}[H]
\begin{center}
\begin{tikzpicture}[scale=0.5]
\draw [gray] (0,0) circle(4cm);
\draw[magenta, fill=blue, fill opacity=0.1] (0,0) circle(2.2cm);
\draw[gray]  (0,0) circle(2cm);
\draw[ cyan,thick,fill=white]  (0,0) circle(2cm);
\draw[magenta,fill=blue, fill opacity=0.1]  (0,0) circle(1.8cm);
\draw(0,0)node[scale=0.7] {$\mathbb A_\delta$};
\draw [->] (2.6,0)--(2.1,0);
\draw(3,0)node[scale=0.7] {$\mathbb A_\delta$};
\end{tikzpicture}
\caption{An example of $\mathbb A_\delta$}
\end{center}
\end{figure}
Since $r_\delta'$ is defined by the relation 
$$\left|\left\{ r^*\leq |x|\leq r^*+r_\delta'\right\}\right|=\frac\delta2$$ the set  $\mathbb A_\delta$ is uniquely defined.
We claim the following:
\begin{claim}\label{Cl:CharacterizationRadial}
There exists $\overline \delta>0$ such that, for any $\delta\leq \overline \delta$, 
$$\mathcal H_\delta=\chi_{\mathbb A_\delta}.$$
\end{claim}
\begin{proof}[Proof of Claim \ref{Cl:CharacterizationRadial}]
To prove this claim, we need the optimality conditions associated with \eqref{Eq:PvDeltaRadial}. We first note that, if $u_\delta$ is the eigenfunction associated with $\mathcal H_\delta$, there exist two real numbers $\eta_\delta,\mu_\delta$ such that 
\begin{enumerate}
\item\label{Interieur} $\mathcal H_\delta=\chi_{\{u_\delta>\mu_\delta\}\cap \B^*}$ in $\B^*$,
\item $\mathcal H_\delta=\chi_{\{\sup_{\mathbb S^*} u_\delta\geq  u_\delta>\eta_\delta\}\cap (\B^*)^c}$ in $(\B^*)^c$,
\item $\left|\{u_\delta>\mu_\delta\}\cap \B^*\right|=V_0-\frac{\delta}2$, $\left|\{\sup_{\mathbb S^*} u_\delta\geq  u_\delta>\eta_\delta\}\cap (\B^*)^c\right|=\frac\delta2$.
\end{enumerate}
This is readily seen from the Rayleigh quotient formulation \eqref{Eq:Rayleigh}. We only prove \ref{Interieur}: let $\mu_\delta \in \R$ be the only real number such that 
$$ \left|\{u_\delta>\mu_\delta\}\cap \B^*\right|=V_0-\frac{\delta}2=\int_{\B^*}\mathcal H_\delta$$
and replace $\mathcal H_\delta$ by 
$$\tilde{\mathcal H_\delta}:=\chi_{ \{u_\delta>\mu_\delta\}\cap \B^*}+\mathcal H_\delta\chi_{(\B^*)^c}.$$ Since $u_\delta$ is radially symmetric (because $\mathcal H_\delta$ is radially symmetric), $\tilde{ \mathcal H}_\delta$ is radially symmetric.
\\Then, because $\int_{\B^*}\tilde{\mathcal H}_\delta=\int_{\B^*}\mathcal H_\delta$,  we have, by the bathtub principle (see \cite{HenrotPierre})
$$\int_{\B^*} u_\delta^2\mathcal H_\delta\leq \int_{\B^*}u_\delta^2\chi_{ \{u_\delta>\mu_\delta\}\cap \B^*},$$
hence 
$$\lambda(\mathcal H_\delta)\geq \int_\B |\n u_\delta|^2-\int_\B \tilde{\mathcal H}_\delta u_\delta^2\geq \lambda\left(\tilde{\mathcal H}_\delta\right).$$ This gives the required property.
\\We now need to exploit these optimality conditions. First of all, since $u_\delta$ is radially symmetric, we can define 
$$\zeta_\delta:=\left.u_\delta\right|_{\mathbb S^*}.$$
By standard elliptic estimates, we also have 
$$u_\delta \underset{\delta\to 0}\rightarrow u_*\text{ in }\mathscr C^{1,s} \text{ ( $s<1$)},$$ where $u_*$ is the eigenfunction associated with $V^*$. Since $\frac{\partial u_*}{\partial r}<-c$ on $\{|x|\geq \e\}$, $u_\delta$ is radially decreasing in $\{|x|>\e\}$ for $\delta>0$ small enough.
 It follows that $\mu_\delta>\zeta_\delta$ for $\delta$ small enough. For the same reason, $\zeta_\delta>\eta_\delta$ for $\delta$ small enough. Hence we have
$$\mu_\delta>\eta_\delta>\eta_\delta\, ,\mathcal H_\delta=\chi_{\{u_\delta>\mu_\delta\}}+\chi_{\{\zeta_\delta\geq u_\delta>\eta_\delta\}}.$$
Finally, once again because $u_\delta$ is radially decreasing on $\{|x|>\e\}$ for $\delta$ small enough, both level sets $\{u_\delta>\mu_\delta\}$ and $\{\zeta_\delta\geq u_\delta>\eta_\delta\}$ are connected, and $\mathcal H_\delta$ is the characteristic function of  a centered ball and of an annulus, i.e 
$$\mathcal H_\delta=\chi_{\{\|x\|\leq r^*-z_\delta\}}+\chi_{\{r^*\leq \|x\|\leq r^*+y_\delta\}}.$$
Since 
$$\left|\mathbb A_\delta \Delta \B^*\right|=\int_\B \left|\mathcal H_\delta-V^*\right|$$ there holds
$$\mathcal H_\delta=\chi_{\mathbb A_\delta}$$ for $\delta$ small enough, as claimed.

\end{proof}
We now turn to the proof of \eqref{Eq:RadialRate}: since $\mathcal H_\delta$ is the  minimizer of $\lambda$ in $\tilde{\mathcal M}_\delta$ we are going to prove that there exists a constant $C>0$ such that 
\begin{equation}\label{Eq:RateH}
\lambda(\mathcal H_\delta)\geq \lambda(V^*)+C\delta^2\end{equation}
for $\delta\leq \overline \delta$. Because of Lemma \ref{Le:L1Neigh}, \eqref{Eq:RadialRate} will follow. To prove \eqref{Eq:RateH}, we use parametric derivatives. Let us fix notations:
\begin{enumerate}
\item For any $\delta>0$ small enough so that $\mathcal H_\delta=\chi_{\mathbb A_\delta}$, we set 
$$h_\delta:=\mathcal H_\delta-V^*.$$
\item For any $t\in [0;1]$ we define $V_{\delta,t}$ as
$$V_{\delta,t}:=V^*+th_\delta$$ and  $u_{\delta,t}$ as the eigenfunction associated with $V_{\delta,t}$:
\begin{equation}\label{Eq:EigenFunctionT}
\left\{\begin{array}{ll}
-\Delta u_{\delta,t}-V_{\delta,t}u_{\delta,t}=\lambda_tu_{\delta,t}\text{ in }\O,&
\\u_{\delta,t}=0\text{ on }\partial \O.&
\\\int_\O u_{\delta,t}^2=1.
\end{array}
\right.
\end{equation}
\item For any such $\delta$, $\dot u_\delta$ is the  first order parametric derivative of $u$ at $V_{\delta,t}$ in the direction $h_\delta$ and $\lambda_{\delta,t}$ is the first order parametric derivative of $\lambda$ at $V_{\delta,t}$ in the direction  $h_\delta$.\end{enumerate}
As is proved in Annex \ref{An:Diff}, these objects are well-defined. Differentiating the equation  with respect to $t$ gives
\begin{equation}\label{Eq:FirstParametricDerivative}
\left\{\begin{array}{ll}
-\Delta \dot u_{\delta,t}=\lambda(V^*) \dot u_{\delta,t}+V^* \dot u_{\delta,t}+\dot \lambda_{\delta,t} u_*+h_{\delta,t} u^*\, , &\text{ in }\O,
\\\dot u_{\delta,t} =0\text{ on }\partial \O,&
\\\int_\B u_*\dot u_{\delta,t}=0,&
\end{array}
\right.
\end{equation}
and
, multiplying the first equation by $u_{\delta,t}$ and integrating by parts ,
$$\dot \lambda_{\delta,t}=-\int_\B h_{\delta} u_{\delta,t}^2.$$
We apply the  mean value Theorem  to 
$f:t\mapsto \lambda_{\delta,t}.$
This gives the existence of $t_1\in [0;1]$ such that
$$\lambda(\mathcal H_\delta)-\lambda(\B^*)=f(1)-f(0)=f'(t_1)=-\int_{\B}h_\delta u_{\delta,t_1}^2.$$
 Our goal is now  to prove that 
\begin{equation}\label{Eq:Controle0}
-\int_\B h_\delta u_{\delta,t_1}^2\geq C\delta^2\end{equation} for some constant $C>0$ whenever $\delta$ is small enough. We will actually prove the existence of $\underline \delta>0$ such that, for any $t\in [0;1]$ and any $\delta\leq \underline \delta$, there holds
\begin{equation}\label{Eq:Controle1}
-\int_\B h_\delta u_{\delta,t}^2\geq C\delta^2\end{equation} for some $C>0$.
\begin{proof}[Proof of Estimate \eqref{Eq:Controle1}]
We recall that $r_\delta$ and $r_\delta'$ were defined in      \eqref{Eq:ADelta}-\eqref{Eq:ADeltaAd}. We can rewrite $\mathcal H_\delta=\chi_{\mathbb A_\delta}\in X_\delta$ under the form
$$\int_{r^*-r_\delta}^{r^*}t^{n-1}dt+\int_{r^*}^{r^*+r_\delta'}t^{n-1} dt=\frac{\delta}{c_n}$$ where $c_n=\mathcal H^{n-1}(\mathbb S(0;1))$ and the condition 
$\int_\O h_\delta=0$ implies
$$\int_{r^*-r_\delta}^{r^*}t^{n-1}dt=\int_{r^*}^{r^*+r_\delta'}t^{n-1} dt.$$
An explicit computation yields the existence of a constant $C>0$ such that
\begin{equation}\label{Eq:EquivDelta}r_\delta, r_\delta'\underset{\delta \to 0}\sim C\delta.\end{equation}
Let $I_\delta^\pm:=\{h_\delta=\pm1\}.$ Since $h_\delta$ is radial, for any $t\in [0;1]$ the function $u_{\delta,t}$ is radial. 
\paragraph{First facts regarding $u_{\delta,t}$}
Identifying $u_{\delta,t}$ (resp. $V_{\delta,t}$) with. the unidimensional function $\tilde u_{\delta,t}$  (resp. $\tilde V_{\delta,t}$) such that
$$u_{\delta,t}(x)=\tilde u_{\delta,t}(|x|) \text{ (resp. $V_{\delta,t}(x)=\tilde V_{\delta,t}(|x|)$)}$$ we have the following equation on $u_{\delta,t}$:
\begin{equation}\label{Eq:EigenFunctionUniT}
\left\{\begin{array}{ll}
-\frac1{r^{n-1}}\left(r^{n-1} u_{\delta,t}'\right)'=V_{\delta,t}u_{\delta,t}+\lambda_tu_{\delta,t}\text{ in }[0;R],&
\\u_{\delta,t}'(0)=u_{\delta,t}(1)=0,&
\\\int_0^1 xu_{\delta,t}(x)^2dx=\frac1{c_n}.$$
\end{array}
\right.
\end{equation}
Since $V_{\delta,t}$ is constant in $ (r^*-r_\delta;r^*)\cup (r^*;r^*+r_\delta')$ and since $u_{\delta,t}$ is uniformly bounded in $L^\infty$ by standard elliptic estimates, $u_{\delta,t}$ is $\mathscr C^2$ in $(r^*-r_\delta;r^*)\cup (r^*;r^*+r_\delta')$. Furthermore, Equation \eqref{Eq:EigenFunctionUniT} readily  gives the existence of a constant $M$ such that, uniformly in $\delta$ and in $t\in [0;1]$, 
\begin{equation}\label{Eq:DeriveeSecondeT}
\left|\left| u_{\delta,t}\right|\right|_{W^{2,\infty}}\leq M.
\end{equation}
Finally, it is standard to see that Equation \eqref{Eq:EigenFunctionT}  gives
\begin{equation}\left|\left| u_{\delta,t}-u_*\right|\right|_{\mathscr C^1}\underset{\delta \to 0}\rightarrow 0\end{equation} uniformly in $t$. As a consequence, since 
$$u_*'(r^*)<0$$ there exists $\underline \delta_1>0$ such that, for any $\delta\leq \underline \delta_1$ and any $t\in [0;1]$, 
\begin{equation}\label{Eq:ControleDeriveeT}
u_{\delta,t}'(r^*)\leq -C<0.
\end{equation}
\paragraph{End of the Proof}
For any $x\in I_\delta^\pm$ and any $t\in [0;1]$, a Taylor expansion gives 
$$u_{\delta,t}^2(x)=\left.u_{\delta,t}^2\right|_{\mathbb S^*}\mp\left.2u_{\delta,t}|\n u_{\delta,t}|\right|_{\mathbb S^*} dist(x;\mathbb S^*)+o\left(dist(x;\mathbb S^*)\right),$$ and $o\left(dist(x;\mathbb S^*)\right)$ is uniform in $\delta>0$ small enough and $t\in [0;1]$ by Estimate \eqref{Eq:DeriveeSecondeT}. This Taylor expansion gives
\begin{align*}
\dot\lambda_{\delta,t}&=-\int_\B h u_{\delta,t}^2&
\\&=-\left. u_{\delta,t}^2\right|_{\mathbb S^*}\int_\B h_\delta \left(=0\text{ because } \int_\B h=0\right)
\\&+\int_{I_\delta^-}\left.2u_{\delta,t}|\n u_{\delta,t}|\right|_{\mathbb S^*} dist(x;\mathbb S^*)+o\left(\int_{I_\delta^-}dist(x;\mathbb S^*)\right)
\\&+\int_{I_\delta^+}\left.2u_{\delta,t}|\n u_{\delta,t}|\right|_{\mathbb S^*} dist(x;\mathbb S^*)+o\left(\int_{I_\delta^+}dist(x;\mathbb S^*)\right)
\end{align*} where the $o\left(\int_{I_\delta^+}dist(x;\mathbb S^*)\right)$ are uniform in $t\in [0;1]$ and $\delta$. Furthermore, 
$$u_{\delta,t}|\n u_{\delta,t}|_{\mathbb S^*}\geq C>0$$ for some constant $C>0$ independent of $\delta$ and $t$ by Estimate \eqref{Eq:ControleDeriveeT}.
\\Hence
$$\dot \lambda_{\delta,t}\underset{\delta \to 0}\sim  \left.2u_{\delta,t}|\n u_{\delta,t}|\right|_{\mathbb S^*}\left(\int_{I_\delta^+}dist(x;\mathbb S^*)+\int_{I_\delta^+}dist(x;\mathbb S^*)\right).$$
However, 
\begin{align*}
  \int_{I_\delta^+}dist(x;\mathbb S^*)&=\int_{r^*}^{r^*+r_\delta'} t^{n-1}(t-r^*)dt
  \\&=(r^*)^n r_\delta'\left(\frac1{n+1}\begin{pmatrix} n+1\\n\end{pmatrix}-\frac1n\begin{pmatrix}n\\n-1\end{pmatrix}\right)
  \\&+(r^*)^{n-1}(r_\delta')^2\left(\frac1{n+1}\begin{pmatrix} n+1\\n-1\end{pmatrix}-\frac1n\begin{pmatrix}n\\n-2\end{pmatrix}\right)
  \\&+\underset{\delta \to 0}o(\delta^2)\quad \text{ by  \eqref{Eq:EquivDelta}}
  \\&=(r^*)^{n-1}(r_\delta')^2 +\underset{\delta \to 0}o(\delta^2)
  \\&\underset{\delta \to 0}\sim C\delta^2
  \end{align*} for some $C >0$ by \eqref{Eq:EquivDelta}. 
  In the same manner, 
$$\int_{I_\delta^-}dist(x;\mathbb S^*)=\int_{r^*-r_\delta}^{r^*} t^{n-1}(r^*-t)dt\underset{\delta \to 0}\sim C'\delta^2$$
and so, combining these estimates gives  
$$\dot\lambda_{\delta,t}\underset{\delta \to 0}\geq C\delta^2+o(\delta^2)$$ uniformly in $\delta$ and $t$, which concludes the proof.
\end{proof} \end{proof}

\subsection{Step 4: shape derivatives and quantitative inequality for graphs}
\subsubsection{Preliminaries and notations}
In this Subsection, we aim at proving Theorem \ref{Th:Quanti} for $V= \chi_E$ and where $E$ can be obtained as  a normal graph over $\B^*$. 
\paragraph{Introduction of the Lagrangian and optimality conditions}
We introduce the Lagrange multiplier $\tau$ associated with the volume constraint and define the Lagrangian 
$$L_\tau:E\mapsto \lambda(E)-\tau Vol(E).$$ From classical results in the calculus of variations, we have the following optimality conditions.
\begin{claim}The necessary optimality conditions for a shape $E$ to be a local minimizer (however they are not sufficient) are: 
$$\forall \Phi \in \mathcal X_1(E)\,, \lambda'(E)[\Phi]=0\,,\quad \forall \Phi \in \mathcal X_1(E)\,, L_\tau''(E)[\Phi,\Phi]>0.$$\end{claim}
Here, we use the notion of shape derivatives introduced in Definition \ref{De:ShapeStab}. Since we only want a local quantitative inequality for shapes that can be obtained as normal graphs over the ball $\B^*$, we introduce some notations.

\paragraph{Notations}
We consider in this parts functions $g$ belonging to 
$$\mathcal X_0(\B^*)=\left\{g \in W^{1,\infty}(\B^*)\,, ||g||_{L^\infty(\partial \B^*)}\leq 1\,, \int_{\partial \B^*}g=0\right\}.$$
Whenever $g\in \mathcal X_0(\B^*)$,  there exists $\Phi_g\in \mathcal X_1(\B^*)$ such that 
$$\langle \Phi_g,\nu\rangle=g\text{ on }\partial \B^*.$$

The set $\mathcal X_0$  corresponds to a linearization of the volume constraints for normal graphs and can be seen as a subset of the set of admissible perturbations $\mathcal X_1(\B^*)$ defined in Definition \ref{De:ShapeStab}, in the sense that we restrict admissible perturbations to normal graphs. 
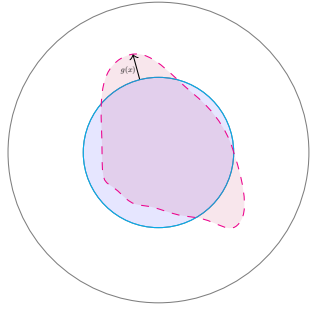
\begin{figure}[H]
\begin{center}
\begin{tikzpicture}[scale=0.5]
\draw [gray] (0,0) circle(4cm);
\draw[gray]  (0,0) circle(2cm);
\draw[ cyan,fill=blue, fill opacity=0.1]  (0,0) circle(2cm);
\draw[magenta,thin,dashed,fill=purple, fill opacity=0.1] plot [smooth, tension=1] coordinates{ (2,0)  (1,1.5)(-1,2.5) (-1.5,0) (-1,-1.1) (0,-1.5)(1,-1.732050)(2.2,-1.8)(2,0)};
\draw[->](-0.5,1.95)--(-0.68,2.61);
\draw(-0.8,2.2)node[scale=0.3] {$g(x)$};
\end{tikzpicture}
\caption{A normal deformation of the ball. The dotted line can be understood as the graph of $g$.}
\end{center}
\end{figure}
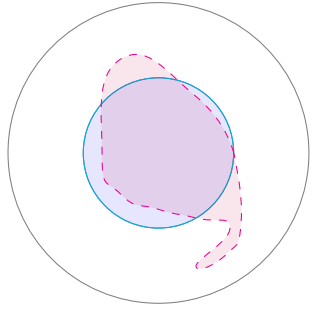
\begin{figure}[H]
\begin{center}
\begin{tikzpicture}[scale=0.5]
\draw [gray] (0,0) circle(4cm);
\draw[gray]  (0,0) circle(2cm);
\draw[ cyan,fill=blue, fill opacity=0.1]  (0,0) circle(2cm);
\draw[magenta,thin,dashed,fill=purple, fill opacity=0.1] plot [smooth, tension=1] coordinates{ (2,0)  (1,1.5)(-1,2.5) (-1.5,0) (-1,-1.1) (0,-1.5)(1,-1.732050)  (1.9,-2) (1,-3) (1.8,-2.7) (2.2,-1.8)(2,0)};
\end{tikzpicture}
\caption{A perturbation of the ball which can not be seen as the graph of a function}
\end{center}
\end{figure}

We define, for any  $g\in \mathcal X_0(\B^*)$ and any $t\in [0;T]$ (with $T$ uniform in $g$ because of the $L^\infty$ constraint) the set 
$\B_{t,g}$ whose boundary is defined as 
$$\partial \B_{t,g}:=\left\{x+tg(x)\nu(x)\,, x \in \partial \B^*\right\}$$ i.e a slight deformation of $E^*$. 
\\We define $$\lambda_{g,t}:=\lambda\left(\B_{t,g}\right).$$
Recall that $\tau$ is the Lagrange multiplier associated with the volume constraint. By defining $L_\tau'(\B^*)[g]:=L_\tau'(\B^*)[\Phi_g]$  and $L_\tau''(\B^*)[g,g]:=L_\tau''(\B^*)[\Phi_g,\Phi_g]$, and with the same convention for other shape functionals involved,  necessary optimality conditions are
\begin{equation}\label{Eq:ShapeOptCond}\forall g\in \mathcal X_0(\B^*)\,, L_\tau'(\B^*)[g]=0\,, L_\tau''(\B^*)[g,g]>0.\end{equation}
We first prove that these optimality conditions hold in the case of a ball and then use them to obtain Theorem \ref{Th:Quanti} for normal perturbations.

\subsubsection{Strategy of proof and comment on the coercivity norm}\label{Su:Strategy}
The strategy of proof is the same as the one used in many articles devoted to quantitative spectral inequalities. For example, we refer to \cite{AcerbiFuscoMoreni,BDPV} for applications of these methods and to the recent \cite{DambrineLamboley}, which presents a general framework for the study of stability and local quantitative inequalities using second order shape variations.
\\Although our method of proof is similar, we point out that the main thing to be careful with here is the coercivity norm for the second-order shape derivative. Indeed, let $J:E\mapsto J( E)$ be a differentiable shape function. In the context of shape spectral optimization, the "typical" coercivity norm at a local minimum $E^*$ for $J$ is the $H^{\frac12}$ norm: in \cite{DambrineLamboley} a summary of shape functionals known to satisfy 
$$J''(E^*)[\Phi,\Phi]\geq C ||\langle \Phi,\nu\rangle||_{H^{\frac12}(\partial E*)}^2$$ is established and proofs of thess coercivity properties are given.  In this estimate, $\Phi$ is an admissible vector field at $E^*$.
\\Here, in the context of parametric shape derivatives, i.e when the shape is a subdomain, it appears (see Subsubsection \ref{Susu:D}) that the natural coercivity norm is the $L^2$ norm:$$L_\tau''(E^*)[\Phi,\Phi]\geq C ||\langle \Phi,\nu\rangle||_{L^{2}(\partial \mathcal A^*)}^2$$  and this coercivity norm is optimal. This makes  things a bit more complicated when dealing with the terms of the second order derivative that involve the mean curvature. This lack of coercivity might be accounted for by the fact that, while in shape optimization, it is the normal derivative of the shape derivative of the eigenfunction that is involved (see \cite{DambrineLamboley}) here, it is just the trace of the shape derivative of the eigenfunction on the boundary of the optimal shape that matters.
\\Once this $L^2$ coercivity is established, we will prove that there exists a constant $\xi,M,C>0$ such that, for any $g$ satisfying $||g||_{W^{1,\infty}}\leq \xi$ and such that the mean curvature of $\mathbb B_{g,t}$ is bounded by $M$ for any $t\leq T$, there holds
\begin{equation}\label{Eq:Valloire}\left| L_\tau''(\B_{t,g})[\Phi_g,\Phi_g]-L_\tau''(\B^**)[\Phi_g,\Phi_g]\right|\leq (C+M)||g||_{W^{1,\infty}} ||g||_{L^2}^2.\end{equation} We then apply the Taylor-Lagrange formula to $f:t\mapsto L_\tau(\B_{t,g})$ to get the desired conclusion, see Subsubsection \ref{Susu:ConclStep4}.

\subsubsection{Analysis of the first order shape derivative at the ball  and computation of the Lagrange multiplier}
The aim of this section is to prove the following Lemma:
\begin{lemma}\label{Le:CriticalityBall}
$\B^*$ is a  critical shape and the Lagrange multiplier associated with the volume constraint is 
$$\tau=-u_*^2|_{\partial \B^*}.$$\end{lemma}
\begin{proof}[Proof of Lemma \ref{Le:CriticalityBall}]
We recall (see \cite{HenrotPierre}) that  
$$Vol'(\B^*)[g]=\int_{\partial \B^*}g\left(=0\text{ if } g\in \mathcal X_0(\B^*)\right).$$
We now compute the first order shape derivative of $\lambda$. The shape differentiability of $\lambda$  follows from an application of the implicit function Theorem of Mignot, Murat and Puel, \cite{MignotMuratPuel}, and  is proved in Appendix \ref{An:Diff}.
\\Let, for any $g\in \mathcal X_0(\B^*)$, $u'_g$ be the shape derivative of the $u_{\B_{t,g}}$ at $t=0$. We recall that this derivative is defined as follows (see \cite[Chapitre 5]{HenrotPierre} for more details): we first define $v_t:=u_{g,t}\circ \Phi_{g,t}$, we define $\dot u_g$ as the derivative in $W^{1,2}_0(\O)$ of $v_t$ with respect to $t$ at $t=0$, and set 
$$u'_g:=\dot u_g-\langle \Phi_g,\n u_0\rangle.$$
We proceed formally to get the equation on $u'_g$ (for rigorous computations we refer to Appendix \ref{An:Diff}): we first differentiate the main equation 
$$-\Delta u_{g,t}=\lambda_{g,t} u_{g,t}+V_{g,t} u_{g,t}$$ with respect to $t$, yielding
$$-\Delta u_g'=\lambda_g' u_*+\lambda_* u_g'+(V^*) u_g'.$$
We then differentiate the continuity equations to get the jump conditions: if we define $$[f](x):=\lim_{y\to x,y\in (\B^*)^c} f(y)-\lim_{y\to x,y\in \B^*} f(y),$$ we have
 $$\left.\left[u_{g,t}\right]\right|_{\partial \B_{t,g}}=\left.\left[\frac{\partial u_{g,t}}{\partial \nu}\right]\right|_{\partial \B_{t,g}}=0,$$
yielding
$$\left.\left[u'_g\right]\right|_{\partial \B^*}=-g\left.\left[\frac{\partial u_{*}}{\partial \nu}\right]\right|_{\partial \B^*}=0$$ because $u_*$ is $\mathscr C^1$, and 
$$\left.\left[\frac{\partial u'_g}{\partial \nu}\right]\right|_{\partial \B^*}=-g\left.\left[\frac{\partial^2 u_*}{\partial \nu^2}\right]\right|_{\partial \B^*}.$$
However, from the Equation on $u^*$ we see that
$$\left.\left[\frac{\partial^2 u_*}{\partial \nu^2}\right]\right|_{\partial \B^*}=u_*|_{\partial \B^*},$$ so that we finally have the following equation on $u'_g$:
\begin{equation}\label{Eq:Derivative1General}
\left\{
\begin{array}{ll}
-\Delta u'_g=\lambda'_g u_*+\lambda_* u_*+(V^*)u_*&\text{ in }\B(0;R),
\\\left[\frac{\partial u_g'}{\partial \nu}\right]=-g u_*|_{\partial \B^*}.&
\end{array}
\right.
\end{equation}
The weak formulation of this equation reads: for any $\p\in W^{1,2}_0(\B)$, 
\begin{equation}\label{Eq:WFD1}\int_\B \langle \n u_g',\n \p\rangle-\int_{\partial \B^*} gu_* \p =\lambda'_g \int_\B \p u_*+\lambda_* \int_\B \p u_*+\int_\B (V^*)\p u_*.\end{equation}
We finally remark that, by differentiating $\int_\B u_{g,t}^2=1$, we get
$$\int_\B u_* u'_g=0.$$
Taking $u_*$ as a test function in \eqref{Eq:WFD1} and using the normalization condition $\int_\B u_*^2=1$ thus gives
\begin{equation}
\lambda'_g=\int_\B \langle \n u_*,\n u_g'\rangle-\lambda_*\int_\B u_g'u_*-\int_\B (V^*) u_*u_g'-\int_{\partial \B^*} g u_*^2.
\end{equation}
However, since $u'_g$ does not have a jump at $\partial \B^*$, we have 
$$\int_\B \langle \n u_*,\n u_g'\rangle-\lambda_*\int_\B u_g'u_*-\int_\B (V^*) u_*u_g'=0$$ by using the Equation on $u_*$.
\\In the end, we get
$$\lambda_g'=-\int_{\partial \B^*} g u_*^2.$$ Since $u_*^2=\mu_*$ is constant on $\partial \B^*$ and $g\in \mathcal X_0(\B^*)$, 
$$\lambda_g'=-\mu_*\int_{\partial \B^*} g=0.$$
This also enables us to compute the Lagrange multiplier: for a function $g\in W^{1,\infty}(\partial \B^*)$ which is no longer assumed to satisfy $\int_{\partial \B^*}g=0$, one must have 
$$L_\tau'[\B^*](g)=0.$$ Indeed, we know, from Lemma \ref{Le:Schwarz}, that $\B^*$ is the unique minimizer of $\lambda$ under the volume constraint.
\\However, the same computations show that 
$$L_\tau'[\B^*](g)=-\mu^*\int_{\partial \B^*} g-\tau \int_{\partial \B^*} g,$$ and the Lagrange multiplier is thus
\begin{equation}\label{Eq:LM}\tau=-\mu^*=-u_*^2|_{\B^*}.\end{equation}
\end{proof}
We now compute the second order shape derivative of $L_\tau$ at any given shape.

\subsubsection{Computation of the second order shape derivative of $\lambda$}
We explained in Subsection \ref{Su:Strategy} that we need to compute the second order derivative at any given shape in order to apply the Taylor-Lagrange formula. Thus, the objectif of this section is the proof of the following Lemma:
\begin{lemma}\label{Le:SecondDerivative}
The second order derivative of the eigenvalue $\lambda$ at a shape $E$ in the direction $\Phi\in \mathcal X_1(E)$ is given by
$$\lambda''(E)[\Phi,\Phi]=-2\int_{\partial E}uu'\langle \Phi,\nu\rangle+2\int_{\partial E}\frac{\partial u}{\partial \nu}\left(\left[\n^2 u[\Phi,\Phi]\right]-\left[\frac{\partial^2u}{\partial \nu^2}\right]\langle \Phi,\nu\rangle^2\right)+\int_{\partial E}\left(-H u^2-2 u\frac{\partial u}{\partial \nu}\right)\langle \Phi,\nu\rangle^2,$$
where $u'$ is defined by Equation \eqref{Eq:Derivative1General} and $H$ is the mean curvature of $E$.

\end{lemma}
\begin{proof}[Proof of Lemma \ref{Le:SecondDerivative}]
To compute $\lambda''(E)[\Phi,\Phi]$, we use Hadamard's second variation formula (see \cite[Chapitre 5, page 227]{HenrotPierre}):  let $K$ be a $\mathscr C^2$ domain, $f(t)$ be a shape differentiable function , then 
\begin{equation}\label{Eq:Hada2}
  \left.\frac{d^2}{dt^2}\right|_{t=0}\int_{K_{\Phi,t}} f(t)=\int_{K} f''(0)+2\int_{\partial K} f'(0)g+\int_{\partial K}\left(H f(0)+\frac{\partial f(0)}{\partial \nu}\right)g^2.\end{equation}
Let $u'$ be the shape derivative of $u_E$ with respect to $t$ and $u''$ the second order shape derivative of $u_E$ with respect to $t$. We successively  apply \eqref{Eq:Hada2} to $E_{\Phi,t}$ and $f(t)=|\n u_t|^2-u_t^2$ and to $(E_{\Phi,t})^c$ and $f(t)=|\n u_t|^2.$
\\Since $\lambda(E_{\Phi,t})=\int_{E_{\Phi,t}} \left(|\n u_t|^2-u_t^2\right)+\int_{E_{\Phi,t}^c}|\n u_t|^2$, this gives
\begin{align}
\lambda''(E)[\Phi,\Phi]=&2\int_\O \langle \n u'',\n u\rangle-2\int_E u''u
\\&+2\int_\O |\n u'|^2-2\int_E (u')^2
\\\label{Li:Jump}&+4\int_{\partial E}\left[\langle \n u,\n u'\rangle\right]\langle \Phi,\nu\rangle-4\int_{\partial E}u'u\langle \Phi,\nu\rangle
\\&+\int_{\partial E}\left(-H u^2-2\frac{\partial u}{\partial \nu}\left[\frac{\partial^2 u}{\partial \nu^2}\right]-2u\frac{\partial u}{\partial \nu}\right)\langle \Phi,\nu\rangle^2.
\end{align}
Let us simplify this expression: First of all the weak formulation of the equation on $u'$ gives
$$2\int_\O |\n u'|^2-2\int_E (u')^2=\underbrace{2\lambda'\int_\O u'u}_{=0\text{ since }\int_\O u'u=0}+2\lambda\int_\O (u')^2-2\int_{\partial E}\left[u'\frac{\partial u'}{\partial \nu}\right].$$
We also note that by differentiating $\int_\O u_t^2=1$ twice with respect to $t$ we get 
\begin{equation}\label{Eq:Nor}\int_\O u''u+\int_\O (u')^2=0.\end{equation}
We note one last simplification to handle Line \eqref{Li:Jump} in the expression for $\lambda''$: we decompose 
$$\n u= \frac{\partial u}{\partial \nu} \nu+\n^\perp u\,, \left\langle\n^\perp u,\nu\right\rangle=0.$$ We adopt the same decomposition for $u'$ and notice that, since $u'$ does not have a jump at $\partial E$,
$$\left[\n^\perp u\right]=\vec{0}.$$ The notation $\vec{0}$ stands for the zero vector in $\R^n$. The same holds true for $u$, and so, since $\frac{\partial u}{\partial \nu}$ has no jump at $\partial E$, we get
$$\left[\langle \n u,\n u'\rangle\right]=\frac{\partial u}{\partial \nu}\left[\frac{\partial u'}{\partial \nu}\right].$$
Finally, using the weak formulation of the equation on $u$ (Equation \eqref{Eq:EigenFunction}) we get
$$2\int_\O \langle \n u'',\n u\rangle-2\int_E u''u=2\lambda \int_\O u''u-\int_{\partial E}[u'']\frac{\partial u}{\partial \nu}.$$
We then only need to compute the jump $[u'']$ at $\partial E$. However, invoking the $W^{2,2}$ regularity of the material derivative $\ddot u$, we get for the shape derivative $u''$: 
$$[u'']=-2\left[\n u'[\Phi]\right]-\left[\n u[D\Phi(\Phi)]\right]-\left[\n^2 u\left[[\Phi,\Phi]\right]\right].$$
We now use the fact that 		$$\left[\n u\right]=\left[\n^\perp u\right]=\left[\n^\perp u'\right]=0\text{ on }\partial E$$ to rewrite
$$[u'']=-2\left[\frac{\partial u'}{\partial \nu}\right]\langle \Phi,\nu\rangle-\left[\n^2 u\left[[\Phi,\Phi]\right]\right].$$
 If we gather these expressions we get 
 \begin{align*}
 \lambda''(E)[\Phi,\Phi]&=2\lambda\int_\O u''u+2\lambda\int_\O (u')^2\left(=0\text{ because of \eqref{Eq:Nor}}\right)
 \\&+\cancel{4\int_{\partial E}\frac{\partial u}{\partial \nu}\left[\frac{\partial u'}{\partial \nu}\right]\langle \Phi,\nu\rangle}+2\int_{\partial E}\frac{\partial u}{\partial \nu}\left[\n^2 u\left[[\Phi,\Phi]\right]\right]
 \\&+2\int_{\partial E} u' u
 \\&-\cancel{4\int_{\partial E} \frac{\partial u}{\partial \nu}\left[\frac{\partial u'}{\partial \nu}\right]\langle \Phi,\nu\rangle}-4\int_{\partial E}u'u\langle \Phi,\nu\rangle
 \\&+\int_{\partial E}\left(-H u^2-2\frac{\partial u}{\partial \nu}\left[\frac{\partial^2 u}{\partial \nu^2}\right]-2 u\frac{\partial u}{\partial \nu}\right)\langle \Phi,\nu\rangle^2
 \\&=-2\int_{\partial E}uu'\langle \Phi,\nu\rangle+2\int_{\partial E}\frac{\partial u}{\partial \nu}\left(\left[\n^2 u[\Phi,\Phi]\right]-\left[\frac{\partial^2u}{\partial \nu^2}\right]\langle \Phi,\nu\rangle^2\right)
 \\&+\int_{\partial E}\left(-H u^2-2 u\frac{\partial u}{\partial \nu}\right)\langle \Phi,\nu\rangle^2
 \end{align*}
\end{proof}

\subsubsection{Analysis of the second order shape derivative at the ball}\label{Susu:D}
The aim of this paragraph is to prove the following Lemma:
\begin{proposition}\label{Le:SecondDerivativeBall}
There exists a constant $C>0$ such that
$$\forall g \in \mathcal X_0(\B^*)\,, L_\tau''[B^*](g,g)\geq C||g||_{L^2(\partial \B^*)}^2.$$
\end{proposition}
We note that the proof of this Lemma relies on a monotonicity principle, which guarantees the weak $L^2$ coercivity. In fact, this is to be the optimal coercivity, in sharp contrast with shape optimization with respect to the boundary of the whole domain $\partial \O$, where the optimal coercivity usually occurs in the $H^{\frac12}$ norm, as noted in Subsection \ref{Su:Strategy}.
\begin{proof}[Proof of Proposition \ref{Le:SecondDerivativeBall}]
We proceed in several steps. We identify $g$ with the normal vector field $\Phi_g$ that can be constructed from $g$, and, to alleviate notations, write $\Phi=\Phi_g$.
\begin{enumerate}
\item \underline{Computation of $\lambda''$}
We use Lemma \ref{Le:SecondDerivative} and first note that, since the vector field $\Phi$ associated with $g$ is normal,
$$\left[\n^2 u[\Phi,\Phi]\right]=\left[\frac{\partial^2u}{\partial \nu^2}\right]\langle \Phi,\nu\rangle^2.$$  In the case of a ball, $H=\frac1{r^*}$. For notational simplicity, we stick to the notation 
$$H^*=\frac1{r^*}.$$
The second derivative  of $\lambda$ becomes
\begin{equation*}
\lambda''(\B^*)[\Phi,\Phi]=\int_{\partial \B^*}\left(-H^*u_*^2-2u_*\frac{\partial u_*}{\partial\nu}\right)\langle \Phi,\nu\rangle^2-2\int_{\partial \B^*}u_*u'\langle \Phi,\nu\rangle.
\end{equation*}
Taking into account the value of the Lagrange multiplier $\tau$ associated with the volume constraint, see Equation \eqref{Eq:LM}, and 
$$Vol''(\B^*)[\Phi,\Phi]=\int_{\partial \B^*}H^*\langle \Phi,\nu\rangle^2$$ we get
\begin{equation}\label{Eq:SecondDerivativeBall}
L_\tau''(\B^*)[\Phi,\Phi]=2\int_{\partial \B^*}-u_*\frac{\partial u_*}{\partial \nu}\langle \Phi,\nu\rangle^2-u_*u'\langle \Phi,\nu\rangle.\end{equation}
  \item \underline{Separation of variables and first simplifications}
  We identify $g$ with a function $g:[0;2\pi]\rightarrow \R$.
  \\We write the decomposition of $g$ as a Fourier series:
  $$g=\sum_{k\in \N}\alpha_k \cos(k\cdot)+\beta_k\sin(k\cdot).$$
  Since $g\in \mathcal X_0(\B^*)$ we have $\alpha_0=0$ and thus 
  \begin{equation}\label{Eq:Fourier}g=\sum_{k= 1}^\infty\alpha_k\cos(k\cdot)+\beta_k\sin(k\cdot).\end{equation}
  We define, for any $k\in \N^*$, $u_k'$ (resp. $w_k'$) as the shape derivative of $u$ with respect to  the perturbation $g=\cos(k\cdot)$ (resp. $g=\sin(k\cdot)$). Since $\B^*$ is a critical shape from Lemma \ref{Le:CriticalityBall}, $u_k'$ satisfies
  \begin{equation*}
  \left\{\begin{array}{ll}
  -\Delta u_k'=\lambda_* u_k'+V^* u_k',\\
  \left[u_k'\right]=-u_*\cos(k\cdot)\text{ on }\partial \B^*\,,\\
  u_k'=0\text{ on }\partial \O.\end{array}\right.
  \end{equation*}
  Since $u_*$ is constant on partial $\B^*$, we can write, in polar coordinates
  $$u_k'(r,\theta)=\psi_k(r)\cos(k\theta)$$ where $\psi_k$ satisfies the following equation (and we identify $V^*$ with the one dimensional function $\tilde V^*$ such that $(V^*)(x)=\tilde V^*(|x|)=\chi_{|x|\leq r^*}$):
  \begin{equation}\label{Eq:PsiK}
  \left\{
  \begin{array}{ll}
  -\frac1r(r\psi_k')'=\left(\lambda_*+V^*-\frac{k^2}{r^2}\right)\psi_k
  \\\left[\psi_k'\right](r^*)=-u_*(r^*)
  \\\psi_k(R)=0.
  \end{array}
  \right.
  \end{equation}
  In the same way, we have 
  $$w_k'(r,\theta)=\psi_k(r)\sin(k\theta).$$
  Whenever $g$ admits the Fourier decomposition \eqref{Eq:Fourier}, the linearity (with respect to $g$) of the equation on $u'_g$ gives
  $$u_g'=\sum_{k=1}^\infty \alpha_k u_k'+\beta_k w_k'.$$ Plugging this in  the expression of $L_\tau''$, see Equation \eqref{Eq:SecondDerivativeBall}, and using the orthogonality properties of $\left\{\cos(k\cdot),\sin(k\dot)\right\}_{k\geq 1}$ finally yields
  $$L_\tau''(\B^*)[g,g]=L_\tau''(\B^*)[\Phi,\Phi]=\sum_{k=1}^\infty\left\{\alpha_k^2+\beta_k^2\right\}u_*|_{\partial \B^*}\left(-\left.\frac{\partial u_*}{\partial \nu}\right|_{\partial \B^*}-\psi_k(r^*)\right).$$ 
  We define the relevant sequence $\{\omega_k\}_{k\in \N^*}$:
  \begin{equation}\label{Eq:OmegaK}
  \forall k\in \N^*\,,\omega_k:=-\left.\frac{\partial u_*}{\partial \nu}\right|_{\partial \B^*}-\psi_k(r^*),\end{equation} so that   \begin{equation}\label{SecondDerivativeFourier}L_\tau''(\B^*)[g,g]=u_*|_{\partial \B^*}\sum_{k=1}^\infty\omega_k \left\{\alpha_k^2+\beta_k^2\right\}.\end{equation}
  Our goal is now the following Lemma:
  \begin{lemma}\label{Le:OmegaPositivity}
  There exists $C>0$ such that 
  $$\forall k\in \N^*\,, \omega_k\geq C>0.$$\end{lemma}
  In order to prove this  Lemma, we use a comparison principle for one-dimensional differential equations.

  \item \underline{Proof of Lemma \ref{Le:OmegaPositivity}: monotonicity principle}
  We will first prove that 
  \begin{equation}\label{Eq:CompOmegaK}
  \forall k\in \N^*\,, \omega_k\geq \omega_1.\end{equation}
  We first note that
  \begin{equation}\label{Eq:PositivitePsi1}
  \psi_1>0\text{ in }(0;R).
  \end{equation}
  Before we prove, let us see  how \eqref{Eq:PositivitePsi1} implies \eqref{Eq:CompOmegaK}. Let, for any $k\geq 2$, $z_k$ be the function defined as 
  $$z_k:=\psi_k-\psi_1.$$
  Since $\psi_k$ can be expressed as $\psi_k=A_kJ_k(\frac{r}{R})$ for $r\leq r^*$ where $J_k$ is the $k$-th Bessel function of first kind, we have $z_k(0)=z_k(R)=0$. Furthermore, $z_k$ satisfies
  \begin{align}
  -\frac1r(rz_k')'&=\left(\lambda_*+V^*-\frac{k^2}{r^2}\right)\psi_k-\left(\lambda_*+V^*-\frac1{r^2}\right)\psi_1&\nonumber
  \\&\leq \left(\lambda_*+V^*-\frac{k^2}{r^2}\right)\psi_k-\left(\lambda_*+V^*-\frac{k^2}{r^2}\right)\psi_1&\text{ because $\psi_1>0$ by \eqref{Eq:PositivitePsi1}}\nonumber
  \\\label{Eq:ZkIneq}&\leq \left(\lambda_*+V^*-\frac{k^2}{r^2}\right)z_k.&
  \end{align} We also have 
  $$[z_k'](r^*)=0.$$
  However, this equation and this no-jump condition imply
  \begin{equation}\label{Eq:ZkNegative}z_k\leq 0.\end{equation}
  For $k$ large enough, this simply follows by a contradiction argument: if $\lambda_*+V^*-\frac{k^2}{r^2}<0$ in $(0;R)$ then, if $z_k$ reached a positive maximum at some interior point $\overline r$, we should have 
  $$0\leq -\frac1{\overline r}(\overline r z_k''(\overline r))\leq  \left(\lambda_*+V^*-\frac{k^2}{\overline r^2}\right)z_k(\overline r)<0,$$ yielding a contradiction.
  \\A proof of \eqref{Eq:ZkNegative} that is valid for all values of $k$ reads as follows: identifying $u_*$ with its one-dimensional counterpart (i.e with the function $\tilde u_*:[0;R]\rightarrow \R$ such that $u_*(x)=\tilde u_*(|x|)$) we define 
  $$p_k:=\frac{z_k}{u_*}.$$
  We notice that $p_k(0)=0$ and that 
  $$[p_k'](r^*)=0.$$ Furthermore, by straightforward computation,  $p_k$ satisfies
  $$-\frac1r(rp_k')'=\frac{-1}{u_*}(rz_k')'+\frac{z_k}{u_*^2}\frac1r(ru_*')'+2\frac{u_*'}{u_*}p_k'.$$ By \eqref{Eq:ZkIneq} and by non-negativity of $u_*$ we get
  $$-\frac1r(rp_k')'\leq -\frac{k^2}{r^2} p_k.$$
  From this it is straightforward to see by a contradiction argument that $p_k$ can not reach a positive maximum at an interior point. It remains to exclude the case $p_k(R)>0$. 
  \\We argue, once again, by contradiction, and assume that $p_k(R)>0$. Since by l'Hospital rules we have 
  $$p_k(r)\underset{r\to R}\sim \frac{z_k'(r)}{u_*'(r)}$$ we must have 
  $z_k'(r)\leq 0.$ However once again by l'Hospital's rule,
  $$p_k'(r)=\frac{u_*z_k'}{u_*^2}-\frac{u_*'z_k}{u_*^2}\underset{r\to 0}\sim \frac12{z_k'}{2u*}<0.$$
  Hence $p_k$ is locally decreasing at $R$, yielding a contradiction. Thus $p_k\leq 0$ and in turn $\psi_k-\psi_1=z_k\leq 0$, completing the proof of \eqref{Eq:CompOmegaK}.
\\The proof of \eqref{Eq:PositivitePsi1} follows from the same arguments: we define 
$$\Psi_1:=\frac{\psi_1}{u_*}$$ and observe that 
$$-\frac1r(r\Psi_1')'=-\frac1{r^2}\Psi_1+2\frac{u_*'}{u_*}\Psi_1'\,, [\Psi_1'](r^*)=-u_*(r^*).$$
We once again argue by contradiction and assume that $\Psi_1$ reaches a negative minimum. From the jump condition at $r^*$, if this maximum is reached at an interior point, it cannot be at $r=r^*$ and the contradiction follows from the Equation. We exclude the case of a negative minimum at $R$ through the same reasons as for $p_k$.
\\It follows that $\psi_k\leq \psi_1$ so that
$$\omega_k-\omega_1=\psi_1(r^*)-\psi_k(r^*)\geq 0.$$
To conclude the proof of Lemma \ref{Le:OmegaPositivity}, it remains to prove that 
\begin{equation}\label{Eq:Omega1Positive}
\omega_1>0.\end{equation}
\begin{proof}[Proof of \eqref{Eq:Omega1Positive}]
We define
$$\Psi:=u_*'+\psi_1.$$
We note that 
$$-\frac1r(r(u_*')')'=\left(\lambda_*+V^*-\frac1{r^2}\right)u_*'\,,\quad  [(u_*')'](r^*)=[u''](r^*)=u_*(r^*).$$
By Hopf's Lemma, $u_*'(R)<0$ and, since $u_*$ is $\mathscr C^2$ in $\B^*$, $u_*'(0)=0.$
\\We get the following equation on $\Psi$:
$$-\frac1r(r\Psi')'=\left(\lambda_*+V^*-\frac1{r^2}\right)\Psi\,,\quad[\Psi'](r^*)=0\,,\quad \Psi(0)=0\,,\quad \Psi(R)<0.$$
Defining
$$\Theta:=\frac{\Psi}{u_*}$$ we get 
$$-\frac1r(r\Theta')'=-\frac1{r^2}\Theta+2\frac{u_*'}{u_*}\Theta'$$ and $\Theta$ can thus not reach a positive maximum at an interior point. Since it is negative at $r=R$ we get $\Theta\leq 0$ in $[0;R]$. Furthermore, it is not identically zero since $\Theta(R)\neq 0$, and the strong maximum principle implies $\Theta<0$ in $(0;R)$. This gives
$$\Theta(r^*)<0$$ or, equivalently
$$\omega_1=-u_*(r^*)\Theta(r^*)>0$$ and this concludes the proof of Lemma \ref{Le:OmegaPositivity}.
\end{proof}
\item \underline{Conclusion of the proof}
To prove Proposition \ref{Le:SecondDerivativeBall}, we simply write 
\begin{align*}
L_\tau''(\B^*)[g,g]&=\sum_{k=1}^\infty \omega_k \left\{\alpha_k^2+\beta_k^2\right\}&\text{ by \eqref{SecondDerivativeFourier}}
\\&\geq C\sum_{k=1}^\infty  \left\{\alpha_k^2+\beta_k^2\right\}& \text{ by Lemma \ref{Le:OmegaPositivity}}
\\&=C||g||_{L^2}^2
\\&=C||\langle \Phi,\nu\rangle||_{L^2}^2.
\end{align*}
\end{enumerate}
The proof of the Proposition is now complete.

\end{proof}

\subsubsection{Taylor-Lagrange formula and control of the remainder}
We now state the main estimate which will enable us to apply the Taylor-Lagrange formula.
\begin{proposition}\label{Pr:ControlRemainder}
Let $M>0$ and $\eta>0$. There exists $s\in (0;1)\, ,\e=\e(\eta)>0$ such that for any $\Phi \in \mathcal X_1(\B^*)$ satisfying 
$$||\Phi||_{\mathscr C^{1,1}}\leq M\,,||\Phi||_{\mathscr C^{1,s}}\leq \e$$ there holds
\begin{equation}\left|L_\tau''(\B_{\Phi})[\Phi,\Phi]-L_\tau''(\B^*)[\Phi,\Phi]\right|\leq \eta ||\langle \Phi,\nu\rangle ||_{L^2}^2.\end{equation}
\end{proposition} 
As mentioned earlier, this will prove a quantitative inequality for sets of bounded curvature that are in a $\mathscr C^{1,s}$ neighbourhood of $\B^*$. Note that working in the $\mathscr C^{1,s}$ norm rather than in a $\mathscr C^{2,s}$ norm will be enough, since elliptic regularity estimates will prove sufficient for our proofs.
\\The proof of this proposition is technical but not unexpected in this context. We postpone the proof to Appendix \ref{An:Prop}.

\subsubsection{Conclusion of the proof of Step 4}\label{Susu:ConclStep4}
Recall that we have defined
$$f(t):=\lambda(\B_{t\Phi}).$$
Then:
\begin{align*}
\lambda(\B_\Phi)-\lambda(\B^*)&=f(1)-f(0)
\\&=f'(0)&\left(=0\text{ because $\B^*$ is critical }\right)
\\&+\int_0^1 (1-t) f''(t)dt
\\&=\frac12f''(0)+\int_0^1(1-t) \left(f''(t)-f''(0)\right)dt
\\&\geq C||\langle \Phi,\nu\rangle||_{L^2}^2-\eta||\langle \Phi,\nu\rangle||_{L^2}^2
\\&\geq\frac{C}2||\langle \Phi,\nu\rangle||_{L^2}^2
\\&\geq C'||\langle \Phi,\nu\rangle||_{L^1}^2&\text{ by the Cauchy-Schwarz Inequality}
\\&=C'\delta^2 &\text{ because $|\B_\Phi\Delta \B^*|=\delta$}.
\end{align*}
whenever $||\n \Phi||_{L^\infty}$ is small enough. This concludes the proof of Step 4.
\subsubsection{A remark on the coercivity norm}
The $L^2$ coercivity established in Proposition \ref{Le:SecondDerivativeBall} is not only sufficient, but also optimal. Indeed, since $\psi_k$ is non-negative in $(0;R)$, we immediately have the bound
$$0\leq \omega_k\leq -\frac{\partial u_*}{\partial \nu}.$$ In other words, the coercivity norm for the second derivative is the $L^2$ (rather than the $H^{\frac12}$) norm of the perturbation. This is due to the fact that here, in the context of parametric optimization, it is the value of the shape derivative $u'$ rather than the value of its normal derivative that is involved in the second order shape derivative. We not that this is in sharp contrast with classical shape optimization, where the optimization is carried out with respect to the whole domain $\O$, and where the coercivity norm is the $H^{\frac12}$ norm, see \cite{DambrineLamboley} and the references therein.

\subsection{Step 5: Conclusion of the proof of Theorem \ref{Th:Quanti}}
We now conclude the proof of Theorem \ref{Th:Quanti}. \color{black} We will argue  by contradiction, but let us first fix some notations.
For any $\delta>0$, let  $\mathcal V_\delta$ be a solution of the variational problem \eqref{Eq:PvDelta}. In the same way we derived the optimality conditions for the radial version of the optimization problem, that is, for the variational problem \eqref{Eq:PvDeltaRadial}, it is easy to see that $\mathcal V_\delta$ is equal to 0 or $1$ almost everywhere and that, furthermore, if $u_\delta$ is the associated eigenfunction, that  there exists two real numbers $\mu_\delta$ and $\eta_\delta$ such that 
$$\left\{\mathcal V_\delta =1\right\}=\Big(\left\{u_\delta\geq \mu_\delta \Big\}\cap \B^*\right)\cup \Big(\left\{u_\delta\geq \eta_\delta\right\}\cap (\B^*)^c\Big).$$ We refer to Figure \color{Plum}4 \color{black} below.

\begin{remark}
We actually expect that 
$$\mathcal V_\delta=\mathcal H_\delta$$ where $\mathcal H_\delta$ was defined in Step 3, at least for $\delta>0$ small enough, in which case Step 4 would prove irrelevant. Put otherwise, we expect the solution to \eqref{Eq:PvDelta} to be a radially symmetric set, given the symmetries properties involved. We were not able to prove this result but we give several numerical simulations in the one-dimensional case, i.e with $\O=(-1;1)$, that back this conjecture up. We plot, for several values of $\delta>0$ and $V_0=0.6$, both $\mathcal V_\delta$ and $V^*-\mathcal V_\delta$:

\begin{figure}[h!]
\begin{center}
\includegraphics[width=6cm]{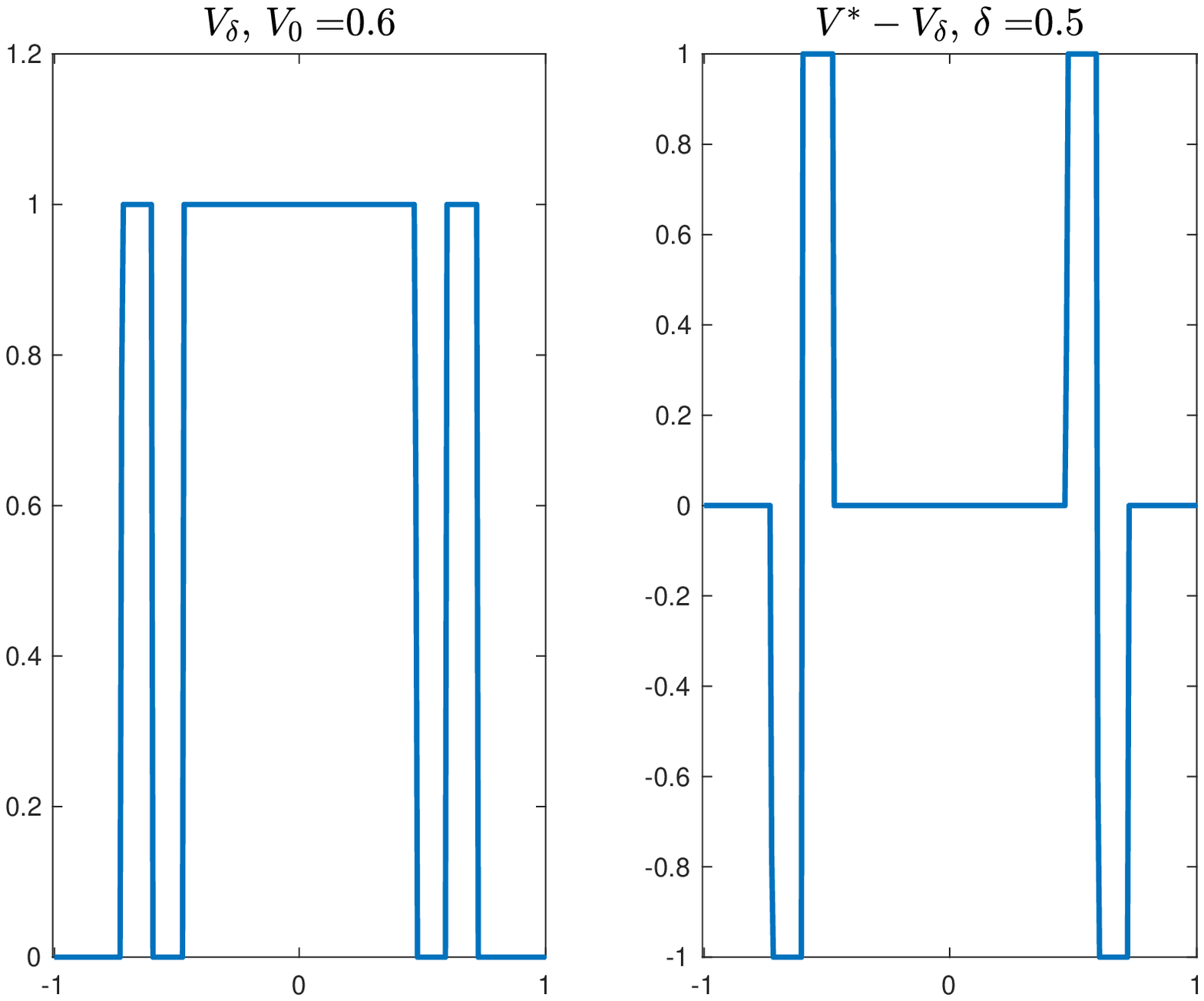}\hspace{1cm}
\includegraphics[width=6cm]{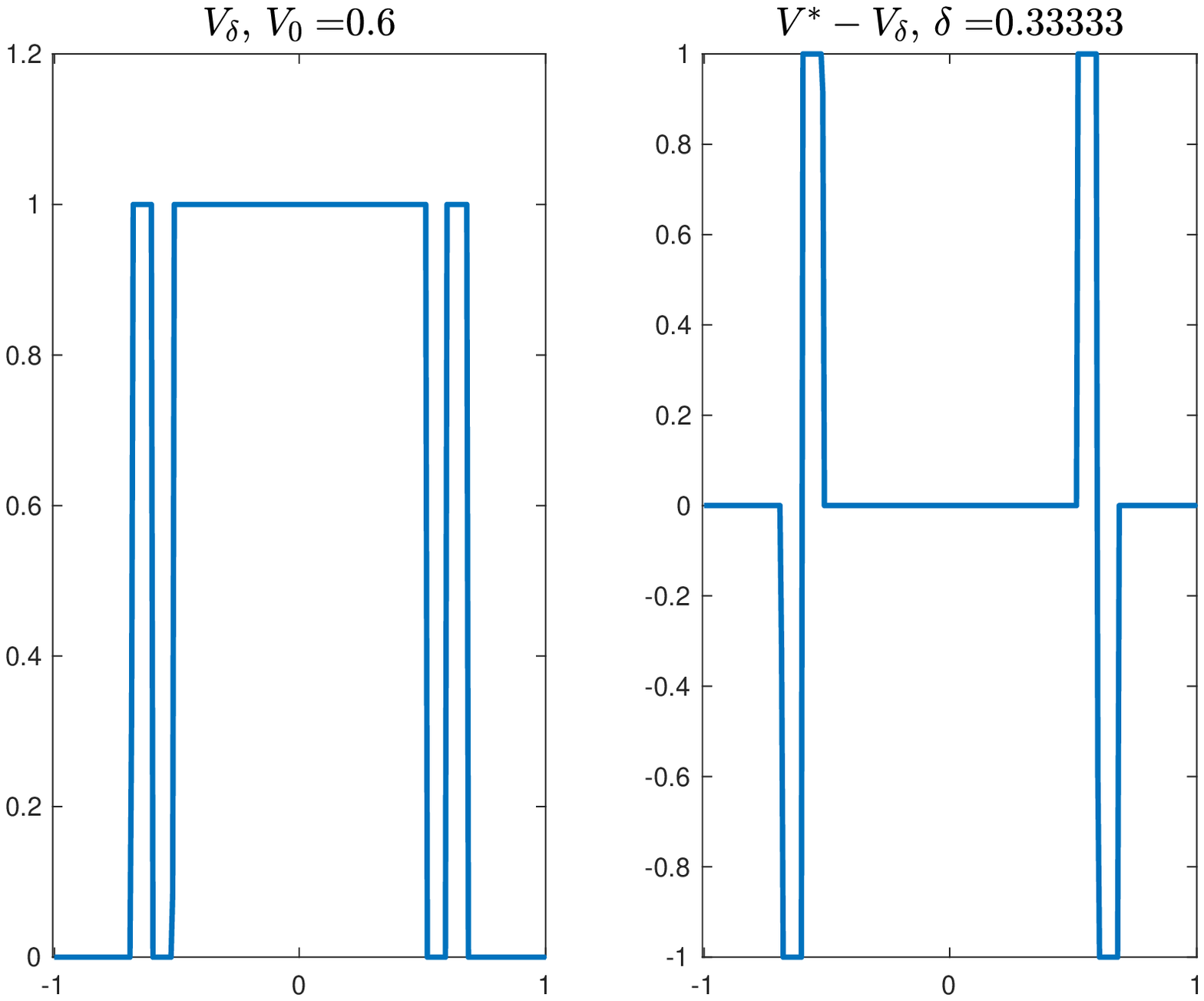}
\end{center}
\end{figure}
\begin{figure}[h!]
\begin{center}
\includegraphics[width=6cm]{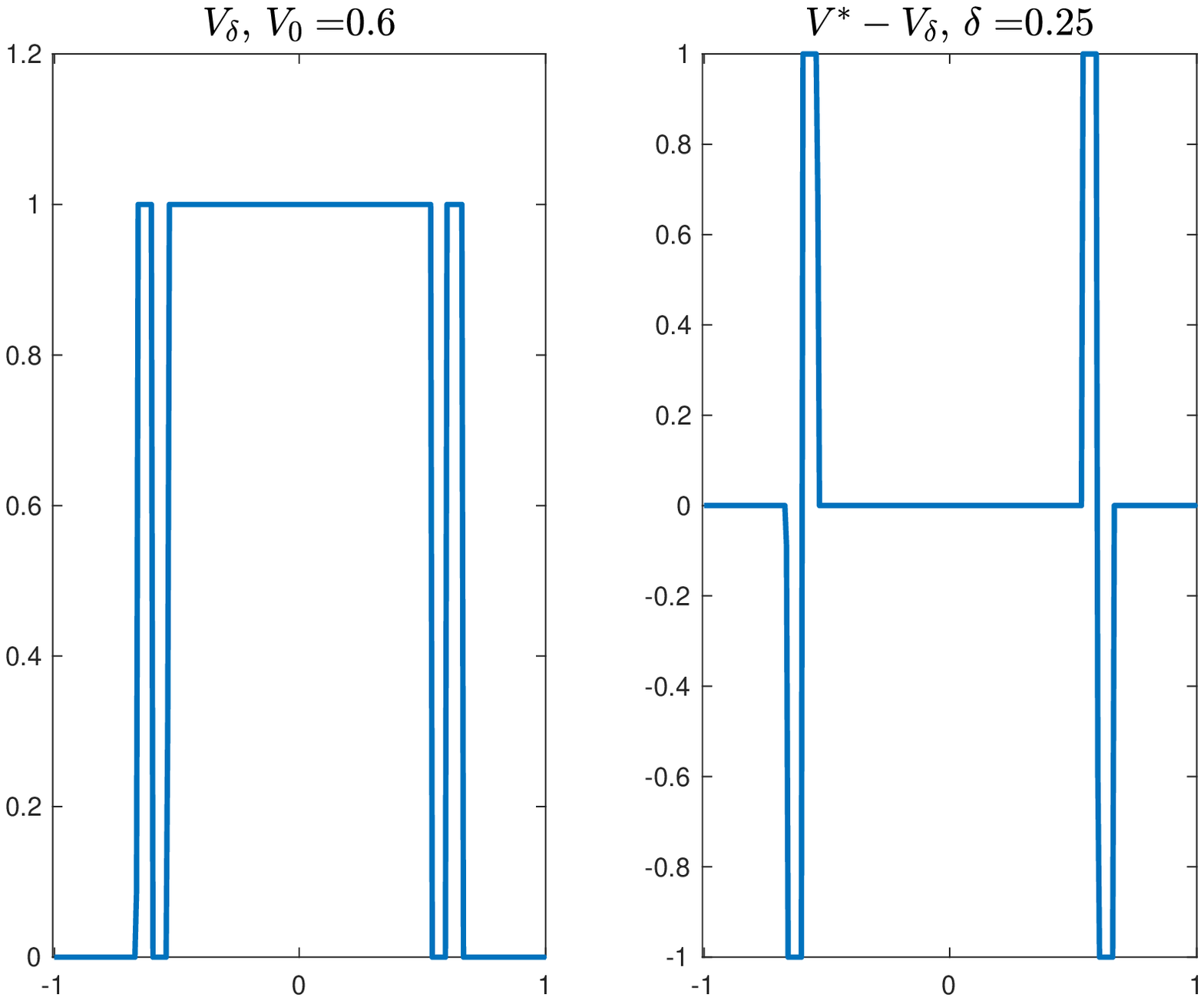}\hspace{1cm}
\includegraphics[width=6cm]{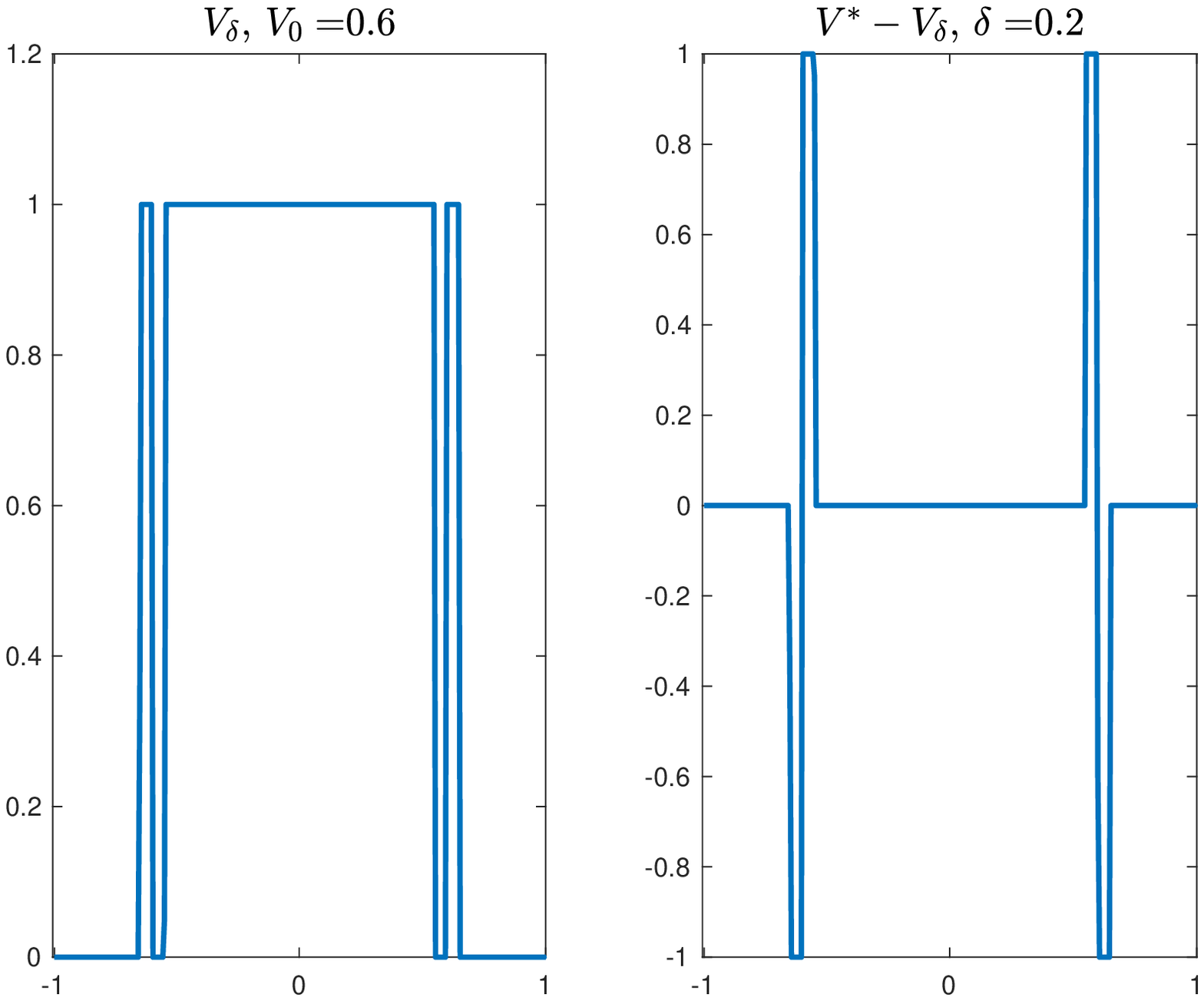}
\end{center}
\end{figure}

\end{remark}

We introduce one last parameter: let $\zeta_{\delta}$ be the unique real number such that
\begin{equation}\label{De:Zeta}\Big|\{ u_\delta>\zeta_\delta\}\Big|=V_0.\end{equation}
In the two figures below, we represent the two most extreme cases we might face (note that we always represent sets that are symmetric with respect to the $x$-axis; this is allowed by Steiner's rearrangement but this property will not be used in what follows)

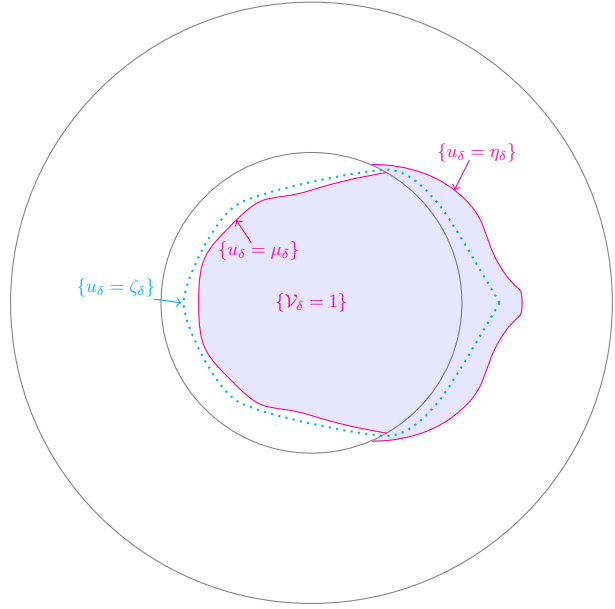
\begin{figure}[H]\label{FigZ}
\begin{center}

\begin{tikzpicture}
\draw [gray] (0,0) circle(4cm);
\draw[magenta,thin,fill=blue, fill opacity=0.1] plot [smooth, tension=1] coordinates{ (0.8,1.84)  (1.9,1.5) (2.5,0.5) (2.8,0) (2.5,-0.5) (1.9,-1.5)(0.8,-1.84)};

\draw[gray,fill=white]  (0,0) circle(2cm);
\draw[gray,fill=blue,opacity=0.1](2,0) arc (0:60:2cm);
\draw[gray,fill=blue,opacity=0.1](2,0) arc (0:-60:2cm);
\fill[blue,opacity=0.1] (1,1.732050)--(2,0)--(1,-1.732050)--(1,1.732050);


\draw[magenta,thin,fill=blue, fill opacity=0.1] plot [smooth, tension=1] coordinates{ (1,1.732050)  (0,1.5)(-1,1.1) (-1.5,0) (-1,-1.1) (0,-1.5)(1,-1.732050)};

\draw[dotted,cyan,thick] plot [smooth, tension=0.4] coordinates{ (2.5,0) (1.7,1.3) (1,1.77) (-1,1.3) (-1.7,0)(-1,-1.3)(1,-1.77)(1.7,-1.3)(2.5,0)};
\draw[->,cyan] (-2.1,0.05)--(-1.73,0);
\draw[cyan] (-2.6,0.2) node [scale=0.7]{ $\{u_\delta=\zeta_{\delta}\}$};
\draw[->,magenta] (-0.8,0.8)--(-1,1.1);
\draw[magenta](-0.7,0.7)node[scale=0.7] {$\{u_\delta=\mu_\delta\}$};
\draw[->,magenta] (2.1,1.9)--(1.9,1.5);
\draw[magenta](2.2,2)node[scale=0.7] {$\{u_\delta=\eta_\delta\}$};

\draw (0:0) node[magenta,scale=0.7] {$\{\mathcal V_\delta=1\}$};

\end{tikzpicture}
\caption{Here, the set $\mathcal V_\delta$ is connected, and we might compare it with a normal deformation of $\B^*$. }
\end{center} 

\end{figure}

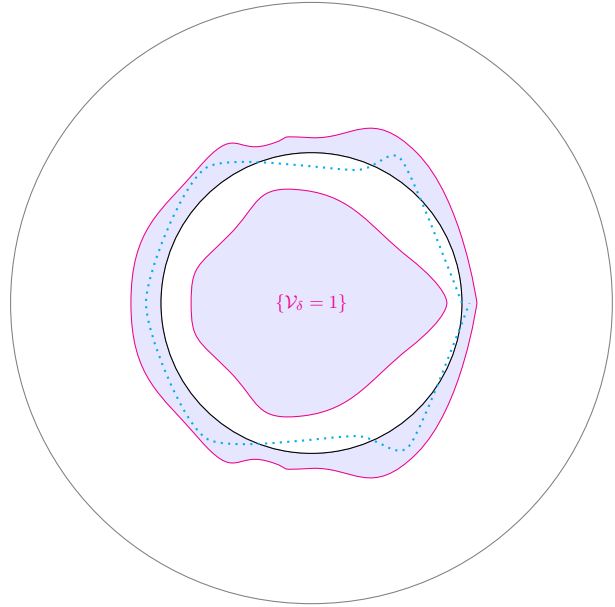
\begin{figure}[H]
\begin{center}
\begin{tikzpicture}
\draw [gray] (0,0) circle(4cm);
\draw[gray]  (0,0) circle(2cm);
\draw[magenta,thin,fill=blue, fill opacity=0.1] plot [smooth, tension=1] coordinates{ (2.2,0) (1.4,2) (0,2.2) (-0.6,2.1) (-1.5,1.8) (-2.4,0) (-1.5,-1.8) (-0.6,-2.1) (0,-2.2) (1.4,-2)(2.2,0)};
\draw[fill=white] circle (2cm);
\draw[magenta,thin,fill=blue, fill opacity=0.1] plot [smooth, tension=1] coordinates{ (1.8,0) (1.3,0.6)  (-0.1,1.5)(-1.1,0.9) (-1.6,0) (-1.1,-0.9) (-0.1,-1.5)(1.3,-0.6) (1.8,0)};
\draw[dotted,cyan,thick] plot [smooth, tension=0.4] coordinates{ (2,0) (1.2,1.9) (0.6,1.77) (-1.4,1.8) (-2.2,0)(-1.4,-1.8)(0.6,-1.77)(1.3,-1.9)(2.1,0)};

\draw (0:0) node[magenta,scale=0.7] {$\{\mathcal V_\delta=1\}$};

\end{tikzpicture}
\caption{Here, the set $\mathcal V_\delta$ is disconnected, and we  might compare it with a radial distribution.}
\end{center} 

\end{figure}

To formalize this, we introduce the quantity
$$f(\delta):=\left|\left\{u_\delta\geq \zeta_\delta\right\}\Delta\left(\{u_\delta\geq \eta_\delta\}\cap (\B^*)^c\right)\right|.$$
Since 
$$\left|\{u_\delta\geq \eta_\delta\}\cap (\B^*)^c\right|=\frac{\delta}2$$ because $\mathcal V_\delta \in \mathcal M_\delta$,  we have
$$f(\delta)\leq \frac{\delta}2.$$

Let us now turn back to the proof of Theorem \ref{Th:Quanti}.

 To prove Theorem \ref{Th:Quanti}, we will as mentioned argue by contradiction:  assume that the estimate \eqref{Eq:Rate} is not valid, that is, there exists a sequence $\{\delta_k\}_{k\in \N}$ such that $$\lim_{k\to \infty}\delta_k=0\text{ and, for any $k\in \N$, $\delta_k>0$}$$ and, furthermore, 
$$\lim_{k\to \infty}\frac{\lambda(\mathcal V_{\delta_k})-\lambda(V^*)}{\delta_k^2}=0.$$  Since we have, for every $k\in \N$, 
$$f(\delta_k)\leq \frac{\delta_k}2$$ we can also assume that, up to an extraction:
$$\text{ the sequence $\left\{ \frac{f(\delta_k)}{\delta_k}\right\}_{k\in \N}$ is converging}.$$ We  now establish a dichotomy depending on the limit of $\left\{ \frac{f(\delta_k)}{\delta_k}\right\}_{k\in \N}$, and  distinguish two cases:
\begin{enumerate}
\item \underline{First case: comparison with a radial distribution}

The first case is defined by 
$$\frac{f({\delta_k})}{\delta_k}\underset{k\to \infty }\rightarrow \ell >0.$$
In that case, $f({\delta_k})\underset{k\to \infty}\sim \ell {\delta_k}$.%
We now apply the bathub principle: let $E_{\delta_k}^1$ be the solution of 
$$\inf_{E\subset \left\{u_{\delta_k}>\zeta_{\delta_k}\right\}\,, |E|=V_0-f({\delta_k})}-\int_{\{u_{\delta_k}\geq \zeta_{\delta_k}\}\cap E} u_{\delta_k}^2.$$
If $\zeta_{{\delta_k},1}$ is defined through
$$\Big|\left\{u_{\delta_k}>\zeta_{{\delta_k},1}\right\}\Big|=V_0-f({\delta_k})$$ then $\zeta_{{\delta_k},1}>\zeta_{\delta_k}$ and consequently
$$E_{\delta_k}^1=\{u_{\delta_k}>\zeta_{{\delta_k},1}\}.$$
In the same way, we define $E_{\delta_k}^2$ as the solution of
$$\inf_{E\subset \Big(\left\{u_{\delta_k}>\zeta_{\delta_k}\right\}\Big)^c\,, |E|=f({\delta_k})}-\int_{\{u_{\delta_k}\geq \zeta_{\delta_k}\}^c\cap E} u_{\delta_k}^2$$ and, if we define $\zeta_{{\delta_k},2}$ through the equation
$$\Big|\left\{\zeta_{\delta_k}>u_{\delta_k}>\zeta_{{\delta_k},2}\right\}\Big|=f({\delta_k})$$ then
$$E_{\delta_k}^2=\left\{\zeta_{\delta_k}>u_{\delta_k}>\zeta_{{\delta_k},2}\right\}.$$
We replace $V_{\delta_k}$ by 
$$\mathcal W_{\delta_k}:=\chi_{E_{\delta_k}^1}+\chi_{E_{\delta_k}^2}.$$ From the bathutb principle,
$$-\int_\O \mathcal V_{\delta_k} u_{\delta_k}^2\geq -\int_\O \mathcal W_{\delta_k} u_{\delta_k}^2.$$ However, $\mathcal W_{\delta_k}$ might not satisfy
$$\int_\O |\mathcal W_{\delta_k}-\chi_{\B^*}|={\delta_k}.$$
We represent $E_{\delta_k}^i$, $i=1,2$, below:

\begin{figure}[H]
\begin{center}

\begin{tikzpicture}[scale=0.6]
\draw [gray] (0,0) circle(4cm);
\draw[magenta,thin,pattern=north west lines, pattern color=gray] plot [smooth, tension=1] coordinates{ (0.8,1.84)  (1.9,1.5) (2.5,0.5) (2.8,0) (2.5,-0.5) (1.9,-1.5)(0.8,-1.84)};
\draw[dotted,cyan,thick,fill=white] plot [smooth, tension=0.4] coordinates{ (2.5,0) (1.7,1.3) (1,1.77) (-1,1.3) (-1.7,0)(-1,-1.3)(1,-1.77)(1.7,-1.3)(2.5,0)};
\fill[blue,opacity=0.1] (2.5,0)--(1.7,1.3)--(1.3,1.7)--(1.15,1.74)--(1,1.75)--(1,-1.75)--(1.15,-1.74)--(1.3,-1.7)--(1.7,-1.3);
\draw[gray,fill=white]  (0,0) circle(2cm);
\draw[gray,fill=blue,opacity=0.1](2,0) arc (0:60:2cm);
\draw[gray,fill=blue,opacity=0.1](2,0) arc (0:-60:2cm);
\fill[blue,opacity=0.1] (1,1.732050)--(2,0)--(1,-1.732050)--(1,1.732050);


\draw[magenta,thin,fill=blue, fill opacity=0.1] plot [smooth, tension=1] coordinates{ (1,1.732050)  (0,1.5)(-1,1.1) (-1.5,0) (-1,-1.1) (0,-1.5)(1,-1.732050)};

\draw[dotted,cyan,thick] plot [smooth, tension=0.4] coordinates{ (2.5,0) (1.7,1.3) (1,1.77) (-1,1.3) (-1.7,0)(-1,-1.3)(1,-1.77)(1.7,-1.3)(2.5,0)};

\end{tikzpicture}
\hspace{3cm}
\begin{tikzpicture}[scale=0.6]
\draw [gray] (0,0) circle(4cm);
\draw[magenta,thin,pattern=north west lines, pattern color=gray] plot [smooth, tension=1] coordinates{ (2.7,0) (1.7,1.4) (1,1.87) (-1.1,1.4) (-1.8,0)(-1.1,-1.4)(1,-1.87)(2.7,0)};
\draw[dotted,cyan,thick, fill=white] plot [smooth, tension=0.4] coordinates{ (2.5,0) (1.7,1.3) (1,1.77) (-1,1.3) (-1.7,0)(-1,-1.3)(1,-1.77)(1.7,-1.3)(2.5,0)};

\draw[magenta,thin,fill=blue, fill opacity=0.1] plot [smooth, tension=1] coordinates{ (2.3,0)(1.1,1.32050)  (0.1,1.5)(-0.9,1.1) (-1.4,0) (-0.9,-1.1) (0.1,-1.5)(1.1,-1.32050) (2.3,0)};
\draw[gray]  (0,0) circle(2cm);

\end{tikzpicture}

\caption{An illustration of the process}
\end{center}
\end{figure}
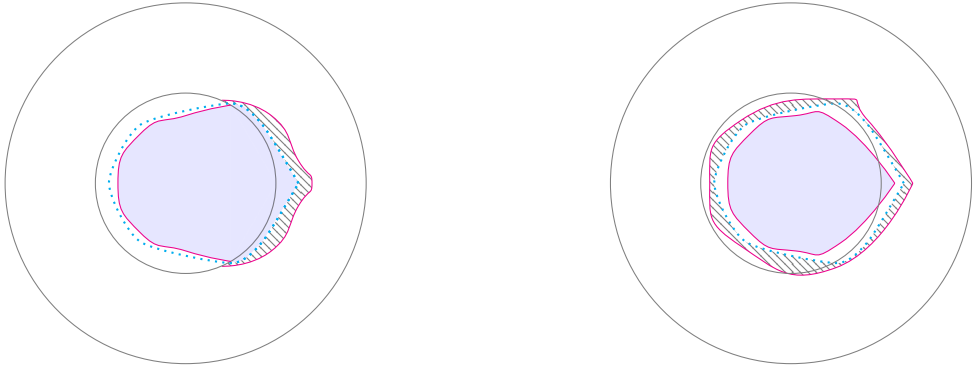
Finally, following the notations of step 3, we recall that $\mathbb A_{f({\delta_k})}$ is defined as 
$$\mathbb A_{f({\delta_k})}:=\left\{|x|\leq r^*-r_{f({\delta_k})}\right\}\cup \{r^*\leq |x|\leq r^*+r_{f({\delta_k})}'\}\,, \chi_{\mathbb A_{f({\delta_k})}}\in \mathcal M_{f({\delta_k})}.$$
Our competitor is $\mathbb A_{f({\delta_k})}$:

\begin{center}
\begin{tikzpicture}[scale=0.6]
\draw [gray] (0,0) circle(4cm);
\draw[magenta, pattern=north west lines, pattern color=gray] (0,0) circle(2.2cm);
\draw[dotted, cyan,thick,fill=white]  (0,0) circle(2cm);
\draw[magenta,fill=blue, fill opacity=0.1]  (0,0) circle(1.8cm);

\end{tikzpicture}\end{center}

Let $u_{\delta_k}^*$ be the Schwarz rearrangement of $u_{\delta_k}$. By equimeasurability of the Schwarz rearrangement, we have 
$$\int_{\mathbb A_{f({\delta_k})}}(u_{\delta_k}^*)^2=\int_\O \mathcal W_{\delta_k} u_{\delta_k}^2.$$ By the Polya-Szego Inequality (see \cite{Kawohl}), 
$$\int_\O |\n u_{\delta_k}^*|^2\leq \int_\O |\n u_{\delta_k}|^2.$$ Finally, we have established the chain of inequalities
\begin{align*}
\lambda(\mathcal V_{\delta_k})&\geq \int_\O |\n u_{\delta_k}|^2-\int_\O \mathcal W_{\delta_k} u_{\delta_k}^2
\\&\geq \int_\O |\n u_{\delta_k}^*|^2-\int_\O \chi_{\mathbb A_{f({\delta_k})}}(u_{\delta_k}^*)^2
\\&\geq \lambda\left(\mathbb A_{(f({\delta_k})}\right)&\text{ by the Rayleigh quotient formulation \eqref{Eq:Rayleigh}.}
\end{align*}
Now, by \eqref{Eq:RadialRate},
$$\lambda\left(\mathbb A_{(f({\delta_k})}\right)\geq \lambda_*+Cf({\delta_k})^2$$ and thus, since $$f({\delta_k})\underset{k\to \infty}\sim \ell {\delta_k}$$ we have 
$$\lambda(\mathcal V_{\delta_k})-\lambda_*\geq C'\ell^2 {\delta_k}^2$$ which  gives the required contradiction and concludes the proof.

\item \underline{Second case: comparison with a normal deformation}

The second case is defined by 
\begin{equation}\label{Eq:SecondCase}\frac{f({\delta_k})}{{\delta_k}}\underset{k\to \infty}\rightarrow 0.\end{equation} In this case, we use Step 4 of the proof, i.e the quantitative inequality for normal deformations of the ball.

Let us replace $\mathcal V_{\delta_k}$ with 
$$\mathcal W_{\delta_k}:=\chi_{\left\{u_{\delta_k}>\zeta_{\delta_k}\right\}}.$$
Recall that $\zeta_{\delta_k}$ was defined in such a way that $\mathcal W_{\delta_k}\in \mathcal M(\O)$. By the bathtub principle, 
$$\lambda(\mathcal V_{\delta_k})\geq \lambda(\mathcal W_{\delta_k}).$$Furthermore, Condition \eqref{Eq:SecondCase} implies
$$\left|\left\{u_{\delta_k}>\zeta_{\delta_k}\right\}\Delta \B^*\right|={\delta_k}+\underset{k\to \infty}o(\delta_k).$$ 
Indeed, 
$$\frac{{\delta_k}}2=\left|\left\{u_{\delta_k}>\eta_{\delta_k}\right\}\cap (\B^*)^c\right|=\left|\left\{u_{\delta_k}>\zeta_{\delta_k}\right\}\cap (\B^*)^c\right|+f({\delta_k}).$$
Finally, standard elliptic estimates imply (in dimension 2 and 3) that
\begin{equation}\label{Eq:CvHolder}u_{\delta_k}\underset{k\to \infty}{\overset{\mathscr C^{1,s}(\overline\O)}\rightarrow}u_*\end{equation} and, since $\frac{\partial u^*}{\partial\nu}|_{\partial \B^*}\neq 0$ and $\zeta_{\delta_k}\underset{k\to \infty}\rightarrow u_*|_{\partial \B^*}$ it follows that 
$\partial \left\{u_{\delta_k}>\zeta_{\delta_k}\right\}$ is a $\mathscr C^1$ hypersurface by the implicit function Theorem.

It remains to prove that $\partial \left\{u_{\delta_k}>\zeta_{\delta_k}\right\}$ is a graph above $\partial \B^*$. 

We start by noticing that \eqref{Eq:CvHolder} implies 
\begin{equation}\label{Eq:ConvHausdorff}d_H(\left\{u_{\delta_k}>\zeta_{\delta_k}\right\}\,, \B^*)\underset{k\to \infty}\rightarrow 0,\end{equation}where $d_H$ is the Hausdorff distance.
We  then argue by contradiction and assume that, a subsequence of $\{{\delta_k}\}_{k\in \N}$ there exists $x_{\delta_k}\in \partial \B^*\,, t_1\neq t_2\in \R$ such that
$$x_{\delta_k}+t_j \nu(x_{\delta_k})\in \partial \left\{u_{\delta_k}>\zeta_{\delta_k}\right\}\,, j=1,2.$$
It follows that 
$$u_\delta(x_{\delta_k}+t_1 \nu(x_{\delta_k}))=u_{\delta_k}(x_{\delta_k}+t_2\nu(x_{\delta_k}))$$ and  by the intermediate value Theorem and \eqref{Eq:ConvHausdorff}, there exists $t_{\delta_k}\in \R$ such that 
$$\langle \n u_\delta(x_{\delta_k}+t_{\delta_k} \nu(x_{\delta_k}))\,,\nu(x_{\delta_k})\rangle=0\,, t_{\delta_k} \underset{k\to \infty}\rightarrow 0.$$ By passing to the limit in this equation up to a subsequence, there exists a point $x_*\in \partial \B^*$ such that 
$$\langle \n u_*(x_*)\,,\nu (x_*)\rangle =0.$$ This is a contradiction since 
$$\left.\frac{\partial u_*}{\partial \nu}\right|_{\partial \B^*}\neq 0.$$
We can then say that $\partial  \left\{u_{\delta_k}>\zeta_{\delta_k}\right\}$ is the graph of a function $\varphi_{\delta_k}$ over $\partial \B^*$. Besides, the Convergence result \eqref{Eq:CvHolder} implies that 
$$\varphi_{\delta_k}\underset{k\to \infty}{\overset{\mathscr C^{1}(\partial \B^*)}\rightarrow }0.$$ Finally, since the set $\{u_{\delta_k}\geq \mu_{\delta_k}\}$ converges in the $\mathscr C^{1,s}$ topology to $\B^*$, there exists a uniform radius $r>0$ such that, for any $x \in \partial \{u_{\delta_k}\geq \mu_{\delta_k}\}$, there exists $y_x$ satisfying
$$x\in \mathbb B(y_x,r)\,, \B(y_x,r)\subset \{u_{\delta_k} \geq \mu_{\delta_k}\}.$$ Since $V_{\delta_k}$ is constant in $\B(y_x,r)$, we can apply elliptic regularity results to get a uniform $\mathscr C^2$ norm on $u_{\delta_k}$ in $\{u_{\delta_k}\geq \mu_{\delta_k}\}$: there exists $M>0$ such that, for any $x\in \{u_{\delta_k}\geq \mu_{\delta_k}\}$, $dist(x\,, \partial \{u_{\delta_k},\mu_{\delta_k}\})\leq r$, $|\n^2u(x)|\leq M$. Hence, the curvature of $\{u_{\delta_k}\geq \mu_{\delta_k}\}$ is uniformly bounded by some constant $M$.

To prove that the curvature of $\{u_{\delta_k}>\zeta_{\delta_k}\}$ is uniformly bounded as well,  we note the following fact: for any $x\in \partial \{u_{\delta_k}>\zeta_{\delta_k}\}$, let $x_{\delta_k}$ be its orthogonal projection on $\{u_{\delta_k} \geq \mu_{\delta_k}\}$.  Then, 
$$\mu_{\delta_k}-\zeta_{\delta_k}\underset{k\to \infty}\sim{|x_{\delta_k}-x|}{\frac{\partial u_{\delta_k}}{\partial \nu}(x_{\delta_k})}$$ and so the map 
$\partial \{u_{\delta_k}\geq \zeta_{\delta_k} \}\ni (x,y)\mapsto\frac{|x-x_{\delta_k}|}{|y-y_{\delta_k}|}$ converges uniformly to $1$. $\{u_{\delta_k}>\zeta_{\delta_k}\}$ can thus be described, asymptotically, as $\{u_{\delta_k}\geq \mu_{\delta_k}\}+\B(0;t_{\delta_k})$, so that it also has a uniformly bounded curvature.
\begin{remark}
We could have worked directly with $\{u_{\delta_k}\geq \mu_{\delta_k}\}+\B(0;t_{\delta_k})$, by choosing a suitable $t_{\delta_k}$ but in this context, it seemed more relevant to work with level sets.
\end{remark}

We can hence apply Step 4:
\begin{align*}
\lambda(\mathcal V_{\delta_k})&>\lambda(\mathcal W_{\delta_k})
\\&\geq \lambda_*+C\left|\left\{u_{\delta_k}>\zeta_{\delta_k}\right\}\Delta \B^*\right|^2\text{ by Step 4}
\\&\geq \lambda_*+C({\delta_k}+\underset{k\to \infty}o({\delta_k}))^2
\\&\geq \lambda_*+C'{\delta_k}^2.\end{align*}

This gives a contradiction.
\end{enumerate}
The proof of Theorem \ref{Th:Quanti} is now complete.
\begin{flushright}$\square$\end{flushright}
\color{black}

\section{Concluding remarks and conjecture}

\subsection{Extension to other domains}
We do believe that this quantitative inequality is valid not only in the ball but for more general domains. Let, for any domain $\O$, $V_\O$ be a solution of \eqref{Eq:PvG}. Let $u_\O$ be the associated eigenfunction. By the bathtub principle, it is easy to see that there exists $\mu_\O\in \R$ such that 
$$V_\O=\chi_{\{u_\O\geq \mu_\O\}}=\chi_{E_\O}.$$
We give the following conjecture:
\begin{conjecture}
Assume that 
\begin{enumerate}
\item The minimizer is regular in the sense that $\frac{\partial u_\O}{\partial \nu}\leq -C<0$ on $\partial E_\O$, 
\item $E_\O$ is a non-degenerate shape minimizer: for any admissible variation $\Phi\in \mathcal X_1(E_\O)$, if $L_\tau$ is the associated lagrangian, there holds
$$L_\tau''(E_\O)[\Phi,\Phi]>0.$$
\end{enumerate}
Then there exists a parameter $\eta>0$ such that, for any $V \in \mathcal M(\O)$, 
$$||V-V_\O||_{L^1(\O)}\leq \eta\Rightarrow \lambda(V)-\lambda(V^*)\geq C ||V-V^*||_{L^1(\O)}^2.$$
\end{conjecture}
Here the main difficulty lies not only in the quantitative inequality for normal perturbations of the domain (Step 4 of the proof of Theorem \ref{Th:Quanti}) but also in the quantitative inequality for possibly disconnected competitors (Step 3 of the proof). Indeed, since the parametric derivatives $\dot u_\O$ are no longer constant on the boundary of the set $E_\O$, the approach used in Step 3 might fail. We note, however, that the infinitesimal quantitative bound
$$\dot \lambda(V_\O)[h]\geq C ||h||_{L^1(\O)}^2$$ still holds. To see why, we notice that 
$$\dot \lambda(V_\O)[h]=-\int_\O hu_\O^2$$ and consider, for a parameter $\delta$, the solution $h_\delta$ of 
$$\min_{h\text{ admissible at $V_\O$}\,,||h||_{L^1(\O)}=\delta}\dot \lambda(V_\O)[h].$$
By the bathtub principle, $h_\delta$ can be written for any $\delta$ as a level set of $u_\delta$ and, for $\delta$ small enough, one can prove that $h_\delta$ writes as follows:
$$h_\delta=\chi_{E_\delta^+}-\chi_{E_\delta^-}\,,$$ where $E_\delta^+\subset E_\O^c$ and $E_\delta^-\subset E_\O$ and can be described as follows: if $\nu$ is the unit normal vector to $E_\O$,
$$E_\delta^\pm:=\{x\pm t \nu (x)\,, t\in (0;t_\pm^\delta(x))\}.$$
We can then prove that 
$$\frac{t_\pm^\delta}{\delta}\underset{\delta \to 0}\rightarrow f_\pm>0$$ uniformly in $x\in \partial E$. It remains to apply the methods of Step 3 of the Proof of Theorem \ref{Th:Quanti} and to do a Taylor expansion of $-u_\O^2$ at $\partial E_\O$ to get
$$\dot\lambda(V_\O)[h_\delta]\geq C \delta^2$$ for some constant $C$ that depends on $\inf \frac{\partial u_\O}{\partial \nu}$. Thus, the infinitesimal inequality seems valid. However, it seems complicated to go further using only this information, since the parametric derivatives $\dot u$ are no longer constant on $\partial E_\O$.
\subsection{Other constraints}
It would be relevant to consider perimeter constraints instead of volume constraints, but we expect the behaviour of the sequences of solutions to the auxiliary problems to be quite different. We nonetheless believe that the free boundary techniques used in \cite{BDPV} might apply directly to get regularity.


\newpage
\appendix

\section{Proof of Lemma \ref{Le:Schwarz}}\label{An:Schwarz}
We briefly recall that the Schwarz rearrangement of a function $u\in W^{1,2}_0(\B)$, $u\geq 0$ is defined as the only radially symmetric non-increasing function $u^*$ such that, for any $t\in \R$, 
$$\left|\{u\geq t\}\right|=\left|\{u^*\geq t\}\right|.$$ 
\begin{proof}[Proof of Lemma \ref{Le:Schwarz}]
We use the Polya-Szego Inequality for the Schwarz rearrangement: for any $u \in W^{1,2}_0(\B)\,, u\geq 0$,
$$\int_{\B}|\n u^*|^2\leq \int_\B |\n u^2|.$$ We also use the Hardy-Littlewood Inequality: for any $u,v\in L^2(\O)$,
$$\int_\B u^*v^*\geq \int_\B uv$$ and the equimeasurability of the rearrangement:
$$\int_\B u^2=\int_\B (u^*)^2.$$We refer to \cite{Kawohl} for proofs.
Using the Rayleigh quotient formulation \eqref{Eq:Rayleigh}, for any $V \in \mathcal M(\B)$,
\begin{align*}
\lambda(V)&=\frac{\int_\B |\n u_V|^2-\int_\B Vu_V^2}{\int_\B u_V^2}
\\&\geq \frac{\int_\B |\n u_V^*|^2-\int_\B V^*(u_V^*)^2}{\int_\B (u_V^*)^2}
\\&\geq \lambda(V^*).\end{align*}
This also proves that $u_{V^*}=u_{V^*}^*$. Since the eigenvalue is simple, the eigenfunction is radially symmetric. The fact that it is decreasing follows from the Equation satisfied by $u_{V^*}$ in polar coordinates.

\end{proof}

\section{Proof of the shape differentiability of $\lambda$}\label{An:Diff}

\subsection{Proof of the shape differentiability}\label{AA}
\begin{proof}[Proof of the shape differentiability]
Let $E$ be a regular subdomain of $\B$,  $(u_0,\lambda_0)$ be the eigenpair associated with $V:=\chi_E$, and let $\Phi$ be an admissible vector field at $E$. Let $T_{\Phi}:=(Id+ \Phi)$ and $E_{\Phi}^*:=T_{\Phi}(E)$. Let $u_{\Phi}$ be the eigenvalue associated with $V_\Phi:=\chi_{E_\Phi}$ and $\lambda_{\Phi}$ be the associated eigenvalue. If we introduce 
$$J_\O(\Phi):=\det(\n T_\Phi)\, , A_\Phi:=J_\O(\Phi)DT_{\Phi}^{-1}\left(DT_{\Phi}^{-1}\right)^t$$ then the weak formulation of the equation on $u_\Phi$ is: for any $v\in W^{1,2}_0(\O)$,
$$\int_\B  \langle A_\Phi\n u_\Phi,\n v\rangle=\lambda_{\Phi} \int_\B u_\Phi v J_\O({\Phi})+\int_\B V_\Phi u_\Phi v J_\O(\Phi).$$
We define the map $F$ in the following way:
$$
F:\left\{\begin{array}{ll}
W^{1,\infty}(\R^n,\R^n)\times W^{1,2}_0(\B)\times \R\to W^{-1,2}(\B)\times \R,&
\\(\Phi, v, \lambda)\mapsto \left(-\n \cdot (-\n \cdot(A_\Phi\n v)-\lambda v J_\O(\Phi)-V_\Phi v J_\O(\Phi), \int_\B v^2J_\O(\Phi)-1\right).&\end{array}
\right.$$
It is clear from the definition of the eigenvalue that 
$$F(0,u_{0},\lambda_{0})=0.$$
Furthermore, the same arguments as in \cite[Lemma 2.3]{DambrineKateb} show that $F$ is $\mathscr C^\infty$ in $B\times W^{1,2}_0(\O)\times \R$, where $B$ is an open ball centered at $\vec 0$.
\\The differential of $F$ at $(0,u_0,\lambda_0)$ is given by 
$$D_{v,\lambda}F(0,u_0,\lambda_0)[w,\mu]=\left(-\Delta w-\mu u_0-\lambda_0 w-Vw, \int_\B 2u_0w\right).$$
To prove that this differential is invertible, it suffices to show that, if $(z,k)\in W^{-1,2}(\O)\times \R$, then there exists a unique couple  $(w,\mu)$ such that 
$$D_{v,\lambda}F(0,u_0,\lambda_0)[w,\mu]=(z,k).$$
By the Fredholm alternative, we know that we must have
$$\mu=-\langle z,u_0\rangle.$$
There  exists a solution $w_1$ to the equation $$-\Delta w-\mu u_0-\lambda_0 w-m^*w=z.$$ We fix  such  a solution. Any other solution is of the form $w=w_1+tu_0$ for a real parameter $t$. We look for such a $t$. From the equation 
$$2\int_\O u_0w=k$$ there comes
$$t=\frac{k}2-\int_\O w_1u_0.$$
hence the couple $(w,\mu)$ is uniquely determined. From the implicit function theorem, the map ${\Phi}\mapsto (u_\Phi,\lambda_{\Phi})$ is $\mathscr C^\infty$  in a neighbourhood of $\vec 0$.
\end{proof}

\subsection{Computation of the first order shape derivative}\label{An:F}
\def\so{{}}
\begin{proof}[Proof of Lemma \ref{Le:SecondDerivative}]
\def\ue{{u_{\e,\alpha}}} 
Let $\Phi$ be a smooth vector field at $E$ and
$E_t=T_t(E)$ where $T_t:=Id+t\Phi$. We define $J_t:=J_\O(t\Phi)$ and $A_t:= A_{t\Phi}$. The other notations are the same as in the previous paragraph.
\\Let $(\lambda_t,u_t)$ be the eigencouple associated with $V_t:= \chi_{E_t}.$ We first define $$v_t:=u_t\circ T_t:\B\rightarrow \R.$$ The derivative of $v_t$ with respect to $t$ will be denoted $\dot u$. This is the material derivative, while we aim at computing the shape derivative $u'$ defined as
$$u'=\dot u+\langle \Phi ,\n u_0\rangle.$$ For more on these notions, we refer to \cite{HenrotPierre}.
\\Obivously $v_t \in W^{1,2}_0(\O)$. The weak formulation on $u_t$ writes:  for any  $\p \in W^{1,2}_0(\O)$, 
$$\int_{\B} \langle \n u_t,\n \p\rangle=\lambda_t \int_{\B} u_t \p+\int_{E_t} u_t \p.$$
We do the change of variables
$$x=T_t(y),$$ so that, for any test function $\p$, 
\begin{equation}\label{Eq:Poitiers}\int_{\B} \langle A_t \n v_t,\n \p\rangle=\lambda_t \int_{\B} v_t\p J_t+ \int_{E} v_t \p J_t.\end{equation}
 It is known that
$$\mathcal J(x):=\left.\frac{\partial J_t}{\partial t}\right|_{t=0}(t,x)=\n \cdot \Phi,$$ and that $$\mathcal A(x):=\left.\frac{\partial A_t}{\partial t}\right|_{t=0}(t,x)=(\n \cdot \Phi)I_n-\left(\n \Phi+(\n \Phi)^T\right).$$

We recall that $\mathcal A$ has the following property: if $\Phi_1$ and $\Phi_2$ are two vector fields, there holds
\begin{equation}\label{PrA}\langle \mathcal A \Phi_1,\Phi_2\rangle=\n \cdot\left(\langle \Phi_1,\Phi_2\rangle \Phi\right)-\langle \n (\Phi\cdot \Phi_1),\Phi_2\rangle-\langle \n (\Phi\cdot \Phi_2),\Phi_1\rangle.\end{equation} 
We differentiate  Equation \eqref{Eq:Poitiers} with respect to $t$ to get the following equation on $\dot u$:
\begin{equation}\label{Eq1}\int_{\B} \langle \n \p,\n \dot u+ \mathcal A \n u_0\rangle=\dot \lambda \int_{\B} u_0 \p+\lambda_0\int_{\B}\mathcal J u_0\p+\lambda_0\int_{\B}\dot u \p+ \int_{E } \dot u \p + \int_{E }\mathcal J(x) u_0\p.\end{equation}
Through Property \eqref{PrA} we get 
\begin{align*}
 \langle \mathcal A \n u_0,\n \p \rangle&= \n \cdot( \langle \n u_0,\n \p\rangle \Phi)- \langle \n (\langle \Phi,\n u_0\rangle),\n \p \rangle- \langle \n (\langle \Phi,\n \p\rangle),\n u_0 \rangle.
\end{align*}
We deal with these three terms separately: from the divergence Formula$$\int_{\B}  \n \cdot( \langle \n u_0,\n \p\rangle \Phi)=-\int_{\partial E}\left[\langle   \n u_0,\n \p\rangle \right]\langle \Phi,\nu\rangle.$$
\\We do not touch the second term.
\\The third term is dealt with using the weak equation on $u_0$:
\begin{align*}
\int_{\B} \langle \n (\langle \Phi,\n \p\rangle), \n u_0 \rangle&=\lambda_0\int_{\B} \langle \Phi,\n \p\rangle u_0+ \int_{E }u_0 \langle \Phi,\n \p\rangle-\int_{\partial E}\left[\langle   \n u_0,\n\p\rangle\right]\langle \Phi,\nu\rangle.
\end{align*}

Hence\begin{align*}
\int_{\B}\langle  \mathcal A \n u_0,\n \p \rangle&=\int_{\B}\so\n \cdot( \langle \n u_0,\n \p\rangle \Phi)-\so\langle \n (\langle \Phi,\n u_0\rangle),\n \p \rangle-\so\langle \n (\langle \Phi,\n \p\rangle),\n u_0 \rangle
\\&=-\int_{\B}\langle  \n (\langle \Phi,\n u_0\rangle),\n \p \rangle
\\&-\lambda_0\int_{\B} \langle \Phi,\n \p\rangle u_0- \int_{E }u_0 \langle \Phi,\n \p\rangle.
\end{align*}
The left hand term of \eqref{Eq1} becomes
\begin{align*}
\int_{\B} \langle \so\n \dot u+\so\mathcal A \n u_0,\n \p\rangle&=\int_{\B}\left\langle \so\n \Big(\dot u-\Phi\cdot \n u_0\Big),\n \p\right\rangle
\\&-\lambda_0\int_{\B} \langle \Phi,\n \p\rangle u_0- \int_{E }u_0 \langle \Phi,\p\rangle.
\end{align*}
Thus\begin{align*}
\int_{\B}\left\langle\so \n \Big(\dot u-\Phi\cdot \n u_0\Big),\n \p\right\rangle
-\lambda_0\int_{\B} \langle \Phi,\n \p\rangle u_0- \int_{E }u_0 \langle \Phi,\n\p\rangle
\\=\dot \lambda \int_{\B} u_0 \p+\lambda_0\int_{\B}\mathcal J(x) u_0\p+\lambda_0\int_{\B}\dot u \p+ \int_{E } \dot u \p + \int_{E }\mathcal J(x) u_0\p.
\end{align*}
By rearranging the terms, we get
\begin{align*}
\int_{\B}\left\langle \so\n \Big(\dot u-\Phi\cdot \n u_0\Big),\n \p\right\rangle
\\=
+\dot \lambda \int_{\B} u_0 \p+\lambda_0\left(\int_{\B}\mathcal J(x) u_0\p
+\int_{\B} \langle \Phi,\n \p\rangle u_0\right)\\+\lambda_0\int_{\B}\dot u \p+ \int_{E } \dot u \p + \int_{E }\mathcal J(x) u_0\p+ \int_{E }u_0\langle \Phi,\n \p\rangle.
\end{align*}
However, since $\mathcal J(x)=\n \cdot \Phi(x)$, we have
$$\mathcal J(x)\p +\langle \Phi,\n \p(x)\rangle=\n \cdot \left(\p \Phi\right).$$
Hence
\begin{align*}
\int_{\B}\mathcal J(x) u_0\p
+\int_{\B} \langle \Phi,\n \p\rangle u_0&=\int_{\B} \n \cdot(\Phi\p)u_0
\\&=-\int_{\B} \p \langle \Phi,\n u_0\rangle,
\end{align*}
because $u_0$ satisfies homogeneous Dirichlet boundary conditions. In the same way
\begin{align*}
 \int_{E }\mathcal J(x) u_0\p
+ \int_{E }\langle \Phi,\n \p\rangle u_0&=\int_{\B} \n \cdot(\Phi\p)u_0
\\&=- \int_{E } \p \langle \Phi,\n u_0\rangle+\int_{\partial E }\langle \Phi,\nu\rangle u_0\p.
\end{align*}
We turn back to the shape derivative; recall that it is  defined as
$$u':=\dot u-\langle \Phi,\n u\rangle.$$
The previous equation rewrites
\begin{align*}
\int_{\B} \langle\so \n u',\n \p\rangle=&\dot \lambda_0\int_{\B} u_0\p+\lambda_0 \int_{\B} u'\p+ \int_{E } u'\p
\\&+\int_{\partial E }\langle \Phi,\nu\rangle u_0\p
\end{align*}
Thus there appears that $u'$ solves
$$-\so \Delta u'=\lambda' u_{0}+\lambda_{0}u_{1}+Vu'{}$$ along with Dirichlet boundary conditions and 
$$\left[\so\frac{\partial u'}{\partial r}\right]=- \langle \Phi,\nu\rangle u_{0}.$$
Obtaining the jump condition on $u''$ is done in the same way as in \cite{DambrineKateb}.
\end{proof}

\subsection{G\^ateaux-differentiability of the eigenvalue}\label{Ap:Differentiability}
The parametric differentiability is also proved using the implicit function theorem applied to the following map:
$$
G:\left\{\begin{array}{ll}
L^\infty(\O)\times W^{1,2}_0(\O)\times \R\to W^{-1,2}(\O)\times \R,&
\\(h, v, \lambda)\mapsto \left(-\Delta v-\lambda v-(V+h)v, \int_\O v^2-1\right).&\end{array}
\right.$$The invertibility of the differential follows from the same arguments as the ones used to prove the invertibility of $DF$ in the previous section.

\section{Proof of Proposition \ref{Pr:ControlRemainder}}\label{An:Prop}

\begin{proof}[Proof of Proposition \ref{Pr:ControlRemainder}]
We can not apply in a straightforward manner the methods of \cite{DambrineLamboley}, which are well-suited for the proof of a convergence in the $H^{\frac12}$ topology. Some minor adjustments are in order.
\\Let us define $T_\Phi:=(Id+\Phi)$ and, for any function $f:\O\rightarrow \R$, 
$$\hat f:=f\circ T_\Phi.$$
We define the surface Jacobian
$$J_\Sigma(\Phi):=\det(\n T_\Phi)\left|\left( ^t\n T_\Phi^{-1}\right)\nu\right|,$$
the volume Jacobian
$$J_\Omega(\Phi):=\det\left(\n \Phi\right)$$ and, finally
$$A_\Phi:=J_\O(\Phi) (Id+\n \Phi)^{-1}(Id+ ^t\n \Phi)^{-1}.$$ 
It is known (see \cite[Lemma 4.8]{DambrineLamboley}) that
\begin{equation}\label{Eq:EstimateJacobian}
\left|\left| J_{\Omega/\Sigma}(\Phi)-1\right|\right|_{L^\infty}\leq C ||\Phi||_{W^{1,\infty}}\,, \left|\left| A_\Phi-1\right|\right|_{L^\infty}\leq C ||\Phi||_{W^{1,\infty}}.\end{equation} We define $u_0$ as the eigenfunction asociated with $\B^*$ and $u_0'$ the shape derivative of $u_0$ in the direction $\Phi$.
\\Finally let $u'_\Phi$ be the shape derivative in the direction $\Phi$ and  $\hat u'_\Phi:=u_\Phi'\circ T_\Phi$. Let $H_\Phi$ be the mean curvature of $\B_\Phi$. Using the change of variable $y=T_\Phi(x)$, the fact that $\Phi$ is normal to $\B^*$ and the value of the Lagrange multiplier $\tau$ given by \eqref{Eq:LM}, we get 
\begin{align*}
L_\tau''(\B_\Phi)[\Phi,\Phi]&=-2\int_{\partial \B^*}J_\Sigma(\Phi)\hat u_\Phi \hat u_\Phi'\langle \Phi,\nu\rangle+\int_{\partial \B^*}J_\Sigma(\Phi)\left(-\hat H {\hat u_\Phi}^2-2\hat u_\Phi\frac{\partial \hat u_\Phi}{\partial \nu}\right)\langle \Phi,\nu\rangle^2
\\&-\tau \int_{\partial \B^*} \hat H\langle \Phi,\nu\rangle^2
\\&=-2\int_{\partial \B^*}J_\Sigma(\Phi)\hat u_\Phi \hat u_\Phi'\langle \Phi,\nu\rangle+\int_{\partial \B^*}J_\Sigma(\Phi)\left(\hat H ({\hat u_\Phi}^2-u_0^2)-2\hat u_\Phi\frac{\partial \hat u_\Phi}{\partial \nu}\right)\langle \Phi,\nu\rangle^2.
\end{align*}
Hence we have 
\begin{align}
\begin{split}\label{Eq:DifferenceSeconde}
L_\tau''(\B_\Phi)[\Phi,\Phi]-L_\tau(\B^*)[\Phi,\Phi]&=-2\int_{\partial \B^*}(J_\Sigma(\Phi)\hat u_\Phi \hat u_\Phi'-u_0u_0')\langle \Phi,\nu\rangle
\\&+\int_{\partial \B^*}\left(\hat H (J_\Sigma(\Phi){\hat u_\Phi}^2-u_0^2)\right)\langle \Phi,\nu\rangle^2
\\&+\int_{\partial \B^*}\left(2u_0\frac{\partial u_0}{\partial \nu}-2J_\Sigma(\Phi)\hat u_\Phi\frac{\partial \hat u_\Phi}{\partial \nu}\right)\langle \Phi,\nu\rangle^2.
\end{split}
\end{align}
We will prove the Proposition using the following estimates
\begin{claim}\label{Cl:EstimateMain}For any $\eta>0$ there exists $\e>0$ such that, for any $\Phi$ satisfying 
$$||\Phi||_{\mathscr C^1}\leq \e$$ there holds
\begin{enumerate}
\item \begin{equation}\label{Eq:1}
||\hat u_\Phi-u_*||_{\mathscr C^1(\O)}\leq \eta,\end{equation}
\item 
 \begin{equation}\label{Eq:Controle}
||\hat u_\Phi'-u_0'||_{W^{1,2}_0}\leq \eta\left|\left| \langle \Phi,\nu\rangle\right|\right|_{L^2(\Sigma)}.\end{equation}\end{enumerate}
\end{claim}

\begin{proof}[Proof of Claim \ref{Cl:EstimateMain}]
Estimate \eqref{Eq:1} follows from a simple contradiction argument and by using the fact that, if a sequence $\{\Phi_k\}_{k\in \N}$ converges in the $\mathscr C^1$ norm to 0, then $u_{\Phi_k}$ converges, in every $\mathscr C^{1,s}(\O)$ ($s<1$) to $u_0$.
\def\NormePhi{{\left|\left| \langle \Phi,\nu\rangle\right|\right|_{L^2(\partial \O)}}}
To prove \eqref{Eq:Controle}, we first prove that there exists a constant $M$ such that 
\begin{equation}\label{Eq:Deri}
||\hat u_\Phi'||_{W^{1,2}_0}\leq M\left|\left| \langle \Phi,\nu\rangle\right|\right|_{L^2(\Sigma)}.\end{equation}
 By the change of variable $y:=T_\Phi(x)$, we see that $\hat u_\Phi'$ satisfies
\begin{equation}\label{Marre}-\n \cdot \Big(A_\Phi \n \hat u_\Phi'\Big)=J_\O(\Phi)\left(\lambda_\Phi \hat u_\Phi' +(V^*) \hat u_\Phi'+\lambda_\Phi' \hat u_\Phi\right)\,, [A_\Phi \partial_\nu \hat u_\Phi']=-J_\Sigma\langle \Phi,\nu\rangle \hat u_\Phi\end{equation} with homogeneous Dirichlet boundary conditions. The orthogonality conditions gives 
$$\int_\O J_\O(\Phi)\hat u_\Phi\hat u_\Phi'=0$$ and we will use a Spectral Gap Estimate \eqref{Eq:SpectralGap}  combined with a bootstrap argument.
\paragraph{Spectral gap estimate}
For any $V\in \mathcal M(\B)$, $\lambda(V)$ was defined as the first eigenvalue of  the operator $\mathcal L_V$ defined in \eqref{Eq:Operator}. We recalled in the Introduction that this eigenvalue is simple. Let, for any $V\in \mathcal M(\O)$, $\lambda_2(V)>\lambda(V)$ and $u_{2,V}$ be the second eigenvalue and an associated eigenfunction (we choose a $L^2$ normalization). We claim there exists $\omega>0$ such that, for any $V\in \mathcal M(\O)$,
\begin{equation}\label{Eq:SG}\omega\leq \lambda_2(V)-\lambda(V).\end{equation} 
To prove this, we use a direct argument. Let $S(V):=\lambda_2(V)-\lambda(V)$ be the spectral gap associated with $V$. We consider a minimizing sequence $\{V_k\}_{k\in \N}\in \tilde{ \mathcal M}(\O)$ (the radiality assumption is not necessary here) which, up to a subsequence, converges weakly in $L^\infty-*$ to some $V_\infty \in \tilde{\mathcal M}(\O).$  It is standard to see that $$\lambda(V_k)\underset{k\to\infty}\rightarrow \lambda(V_\infty)\,, u_{V_k}\underset{k\to \infty}\rightarrow u_{V_\infty} \text{ strongly in $L^2(\B)$, weakly in $W^{1,2}_0(\B)$.}$$ The only part which is not completely classical is to prove that 
\begin{equation}\label{Eq:ConvergenceLambda2}
\lambda_2(V_k)\underset{k\to \infty}\rightarrow \lambda_2(V_\infty).\end{equation}
However, for any $k\in \N$, $\lambda_2(V_k)$ is defined as 
\begin{equation}\label{Eq:Rayleigh2}\lambda_2(V_k)=\min_{u\in W^{1,2}_0(\B)\,,\int_\B u^2=1\,u\in \langle u_{V_k}\rangle^\perp} R_{V_k}[u],\end{equation}  where $\langle u\rangle ^\perp$ is the subspace of functions that are $L^2$-orthogonal to $u$, and $u_{2,V_k}$ is defined as a minimizer for this problem (there a possibly multiple eigenfunctions). In the same way we proved that $\lambda(V)$ is uniformly bounded in $V$, one proves that $\lambda_2(V)$ is uniformly bounded in$V$. Let $\lambda_{2,\infty}$ be such that 
$$\lambda_2(V_k)\underset{k\to \infty}\rightarrow \lambda_{2,\infty}.$$ 
Standard elliptic estimates prove that there exists a function $u_{2,\infty}\in W^{1,2}_0(\O)$ such that 
$u_{2,V_k}\underset{k\to \infty}\rightarrow u_{2,\infty}$ strongly in $L^2(\B)$ and weakly in $W^{1,2}_0(\B)$. Passing to the limit in
$$\int_\B u_{2,V_k}u_{V_k}=0$$ gives 
\begin{equation}\label{Orthogona}\int_\B u_{2,\infty}u_{V_\infty}=0.\end{equation} Passing to the limit in the weak formulation of the equation on $u_{2,V_k}$ proves that $u_{2,\infty}$ is an eigenfunction of $\mathcal L_{V_\infty}$ associated with $\lambda_{2,\infty}$. It follows from the orthogonality relation \eqref{Orthogona} that 
$$\lambda_{2,\infty}\geq \lambda_2(V_\infty).$$ Hence
$$\underset{k\to \infty}{\lim\inf}\text{ }S(V_k)\geq \lambda_2(V_\infty)-\lambda(V_\infty)\geq \omega_1>0$$ because $\lambda(V_\infty)$ is a simple eigenvalue.
\\As a consequence of the spectral gap estimate \eqref{Eq:SG}, we get the following estimate:
\begin{equation}\label{Eq:SpectralGap}
\forall V \in \mathcal M(\O)\,, \forall u\in \langle u_V\rangle^\perp\,,\omega \int_\O u^2\leq \int_\O |\n u|^2-\int_\O Vu^2-\lambda(V)\int_\B u^2.\end{equation}
Indeed, let $V \in \mathcal M(\O)$ and $u\in \langle u_V\rangle^\perp, u\neq 0$. Then, by the Rayleigh quotient formulation on $\lambda_2(V)$, see Equation \eqref{Eq:Rayleigh2},  
$$\int_\B |\n u|^2-\int_\B Vu^2\geq \lambda_2(V)\int_\B u^2\geq \omega \int_\B u^2+\lambda(V)\int_\B u^2\text{ by \eqref{Eq:SG}},$$ which is exactly the desired conclusion.

\paragraph{Proof of \eqref{Eq:Deri}} First of all, multiplying \eqref{Marre} by $\hat u_\Phi'$ and integrating by parts gives
$$\int_\O A_\Phi|\n \hat u_\Phi'|^2-\int_\O V^* J_\O(\Phi)(\hat u_\Phi')^2=\int_{\partial \B^*}J_\Sigma \hat u_\Phi \hat u_\Phi'\langle V,\nu\rangle.$$
By the Spectral gap estimate, using the fact that eigenfunctions are uniformly bounded  and by continuity of the trace operator we get the existence of a constant $M$ such that
\begin{equation}\label{Eq:Esquire}\int_\O (\hat u_\Phi')^2\leq M||\langle \Phi,\nu\rangle||_{L^2(\partial \B^*)} ||\hat u_\Phi'||_{W^{1,2}_0(\O)}.\end{equation}
We rewrite 
$$||\hat u_\Phi'||_{W^{1,2}_0(\O)}=||\hat u_\Phi'||_{L^2(\O)}+||\n \hat u_\Phi'||_{L^2(\O)}.$$
By the shape differentiability of $E\mapsto (\lambda(E),u_E)$, there exists 
$C$ such that
$$||\n \hat u_\Phi'||_{L^2(\O)}\leq C$$ for any $\Phi$ such that $||\Phi||_{W^{1,\infty}}\leq 1$. 
\\We then let $X:=||\hat u_\Phi'||_{L^2(\O)}$. The estimate \eqref{Eq:Esquire} rewrites
$$X^2\leq M\NormePhi X+MC\NormePhi,$$ from where it follows that there exists $M>0$ such that
$$||\hat u_\Phi'||_{L^2(\O)}\leq M\sqrt{\NormePhi}.$$
We now multiply \eqref{Eq:Deri} by  $\hat u_\Phi'$ and integrate by part. Using the continuitiy of the trace operator, this gives, for some constant $M$,
$$\int_\O |\n \hat u_\Phi'|^2\leq C \NormePhi+\NormePhi||\n \hat u_\Phi'||_{L^2(\O)}+\NormePhi^{\frac32}$$ which in turn yields, using the same arguments, 
$$||\n \hat u_\Phi'||_{L^2(\O)}\leq \sqrt{\NormePhi}.$$ We use this in \eqref{Eq:Esquire}, giving
$$||\hat u_\Phi'||_{L^2(\O)}^2\leq M\left(\NormePhi||\hat u_\Phi'||_{L^2(\O)}+\NormePhi^{\frac32}\right).$$
This yields
$$||\hat u_\Phi'||_{L^2(\O)}\leq M\NormePhi$$ and, finally, from the weak formulation of the equation, 
$$||\n \hat u_\Phi'||_{L^2(\O)}\leq M\NormePhi.$$
\paragraph{Proof of \eqref{Eq:Controle}}
We now turn to the proof of the continuity estimate \eqref{Eq:Controle}, for which we will apply the same kind of bootstrap arguments, combined with a version of the splitting method, see \cite[Lemma 4.10]{DambrineLamboley}.
\\ Let us define $H_\Phi$ as the solution of 
\begin{equation}\label{Eq:Splitting1} 
\left\{
\begin{array}{ll}
-\Delta H_\Phi={V_\Phi} H_\Phi,&
\\\left[\frac{\partial H_\Phi}{\partial \nu}\right]=u_\Phi\langle \Phi,\nu_\Phi\rangle\text{ on }\partial E_t,&
\\H_\Phi=0\text{ on }\partial \O.\end{array}\right.\end{equation}
Then it appears that
$$\lambda_\Phi'=-\int_{\partial E_t} u_\Phi^2\langle \Phi,\nu_\Phi\rangle=\lambda_\Phi \int_\O H_\Phi u_\Phi.$$
We can  prove using the same bootstrap arguments used to prove \eqref{Eq:Deri} that 
$$||\hat H_\Phi||_{W^{1,2}_0(\O)}\leq C\NormePhi.$$ Indeed, multiplying \eqref{Eq:Splitting1} by $H_\Phi$, doing a change of variables and integrating by parts gives
$$\int_\O A_t|\n \hat H_\Phi|^2-\int_\O J_t \hat H_\Phi^2\leq \NormePhi||\hat H_\Phi||_{W^{1,2}_0(\O)}$$ and, by the variational formulation of the eigenvalue, 
$$\int_\O \hat H_\Phi^2\leq \NormePhi ||\hat H_\Phi||_{W^{1,2}_0(\O)}.$$ We then use the same bootstrap argument: we first prove that this implies $||\hat H_\Phi||_{L^2(\O)}\leq M\sqrt{\NormePhi}$ and plug this estimate in the weak formulation of the equation. The conclusion follows.
\\We turn back to \eqref{Eq:Controle}.
\\Let $\pi_\Phi$ be the orthogonal projection on $\langle u_\Phi\rangle^\perp$.  We decompose $u_\Phi'$ as 
$$u_\Phi'=-\pi_\Phi H_\Phi+\xi_\Phi$$ where $\xi_\Phi$ solves
\begin{equation}\label{Eq:Splitting1}
\left\{
\begin{array}{ll}
-\Delta \xi_\Phi=\lambda_\Phi\xi_\Phi+{V_\Phi}\xi_\Phi-\lambda_\Phi\pi_\Phi H_\Phi,&
\\\xi_\Phi=0\text{ on }\partial \O.
\\\int_\O \xi_\Phi u_\Phi=0.\end{array}\right.\end{equation}
Thanks to the Fredholm alternative, such a $\xi_\Phi$ exists and is uniquely defined. 
\\We now prove that 
\begin{equation}\label{Eq:Ht}
\left|\left| \hat H_\Phi-H_0\right|\right|_{W^{1,2}_0(\O)}\leq M\eta \NormePhi\end{equation} for $||\Phi||_{\mathscr C^1}$ small enough. To that end, we define 
$$\mathfrak H_\Phi:=\hat H_\Phi-H_0.$$ Direct computation shows that 
$$-\Delta \mathfrak H_\Phi=(V^*) \mathfrak H_\Phi+(V^*)\hat H_\Phi (J_t-1)+\n \cdot\left((A_t-Id)\n \hat H_\Phi\right)$$ along with Dirichlet boundary conditions and 
$$\left[\frac{\partial \mathfrak H_\Phi}{\partial \nu}\right]=(u_0-J_\Sigma \hat u_\Phi) \langle \Phi,\nu\rangle+\left[\langle(Id-A_t)\n \hat H_\Phi\,, \nu\rangle\right].$$ We proceed in the same fashion: we first multiply the equation on $\mathfrak H_\Phi$ by $\mathfrak H_\Phi$, integrate by parts and use the variational formulation of the eigenvalue to get 
\begin{align*}
||\mathfrak H_\Phi||_{L^2(\O)}^2&\leq ||J_t-1||_{L^\infty}||\hat H_\Phi||_{L^2(\O)}+||A_t-Id||_{L^\infty(\O)}||\n \hat H_\Phi||_{L^2(\O)}||\n \mathfrak H_\Phi||_{L^2(\O)}
\\&+\NormePhi||\mathfrak H_\Phi||_{W^{1,2}_0(\O)}||u_0-J_\Sigma\hat u_\Phi||_{L^\infty(\partial \O)}\end{align*} up to a multiplicative constant. This first gives, using \eqref{Eq:1},
$$||\mathfrak H_\Phi||_{L^2(\O)}\leq M\sqrt{\NormePhi(||\n \Phi||_{L^\infty}+\eta)}.$$ We then apply the same bootstrap method to get the desired conclusion.
\\Finally, we need to show the following estimate, which will conclude the proof:
\begin{equation}\label{Eq:ControleSplitting2}
||\hat \xi_\Phi-\xi_0||_{W^{1,2}_0(\O)}\leq\eta \NormePhi
\end{equation}
for $||\Phi||_{\mathscr C^1}$ small enough. However, this follows from the same arguments as in \cite[Lemma 4.10, Paragraph 3 of the proof]{DambrineLamboley} and from the  bootstrap strategy already used.
\end{proof}
Finally, going back to \eqref{Eq:DifferenceSeconde}, it suffices to use the continuity of the trace to control the terms involving $u_\Phi'$ and Estimates \eqref{Eq:1}-\eqref{Eq:Controle} to conclude the proof of Proposition \ref{Pr:ControlRemainder}.
\end{proof}
\newpage

\bibliographystyle{abbrv}
\bibliography{biblio}
\paragraph{Acknowledgment.}
I. Mazari was partially supported by the Project ''Analysis and simulation of optimal shapes - application to lifesciences'' of the Paris City Hall and by the ANR Project "Optimisation de forme - SHAPO".

The author would like to thank D. Bucur and J. Lamboley for the useful discussions they had with him, and express his gratitude to D. Ruiz-Balet for his help in obtaining the numerical simulations of Remark 3.

The author would also like to warmly thank the anonymous referee for his or her insightful comments and advice.

\end{document}